\documentclass[11pt, a4paper,reqno]{amsart}
\usepackage{amsmath}
\usepackage{amsfonts}
\usepackage{amssymb}
\usepackage{amsthm}
\usepackage[active]{srcltx}

\numberwithin{equation}{section}

\newtheorem{theorem}{Theorem}[section]
\newtheorem{lemma}[theorem]{Lemma}
\newtheorem{proposition}[theorem]{Proposition}
\newtheorem{corollary}[theorem]{Corollary}

\theoremstyle{definition}
\newtheorem{definition}[theorem]{Definition}

\theoremstyle{remark}
\newtheorem{remark}[theorem]{Remark}

\def\kint_#1{\mathchoice%
         {\mathop{\kern 0.2em\vrule width 0.6em height 0.69678ex depth -0.58065ex
                 \kern -0.8em \intop}\nolimits_{\kern -0.4em#1}}%
         {\mathop{\kern 0.1em\vrule width 0.5em height 0.69678ex depth -0.60387ex
                 \kern -0.6em \intop}\nolimits_{#1}}%
         {\mathop{\kern 0.1em\vrule width 0.5em height 0.69678ex depth -0.60387ex
                 \kern -0.6em \intop}\nolimits_{#1}}%
         {\mathop{\kern 0.1em\vrule width 0.5em height 0.69678ex depth -0.60387ex
                 \kern -0.6em \intop}\nolimits_{#1}}}
\def\vintslides_#1{\mathchoice%
         {\mathop{\kern 0.1em\vrule width 0.5em height 0.697ex depth -0.581ex
                 \kern -0.6em \intop}\nolimits_{\kern -0.4em#1}}%
         {\mathop{\kern 0.1em\vrule width 0.3em height 0.697ex depth -0.604ex
                 \kern -0.4em \intop}\nolimits_{#1}}%
         {\mathop{\kern 0.1em\vrule width 0.3em height 0.697ex depth -0.604ex
                 \kern -0.4em \intop}\nolimits_{#1}}%
         {\mathop{\kern 0.1em\vrule width 0.3em height 0.697ex depth -0.604ex
                 \kern -0.4em \intop}\nolimits_{#1}}}

\newcommand{\vp}{\varphi}
\newcommand{\eps}{\varepsilon}

\newcommand{\be}{\begin{equation}}
\newcommand{\ee}{\end{equation}}

\newcommand{\bes}{\begin{equation*}}
\newcommand{\ees}{\end{equation*}}

\newcommand{\R}{\mathbb{R}}
\newcommand{\Rn}{\mathbb{R}^d}

\newcommand{\N}{\mathbb{N}}

\newcommand{\norm}[1]{\left| #1 \right|}

\newcommand{\esssup}{\operatornamewithlimits{ess\, sup}}
\newcommand{\essinf}{\operatornamewithlimits{ess\,inf}}
\newcommand{\essosc}{\operatornamewithlimits{ess\, osc}}

\newcommand{\divt}{\operatorname{div}}
\renewcommand{\div}{\nabla \cdot}

\renewcommand{\l}{\left}
\renewcommand{\r}{\right}

\providecommand{\brc}[1]{\left\lbrace#1\right\rbrace}

\def\Xint#1{\mathchoice
{\XXint\displaystyle\textstyle{#1}}%
{\XXint\textstyle\scriptstyle{#1}}%
{\XXint\scriptstyle\scriptscriptstyle{#1}}%
{\XXint\scriptscriptstyle\scriptscriptstyle{#1}}%
\!\int}

\def\XXint#1#2#3{{\setbox0=\hbox{$#1{#2#3}{\int}$}
\vcenter{\hbox{$#2#3$}}\kern-.5\wd0}}

\def\dashint{\Xint-}

\title[H\"older continuity for Trudinger's equation]{H\"older continuity for Trudinger's equation in measure spaces}
\author[Kuusi, Laleoglu, Siljander and Urbano]{Tuomo Kuusi, Rojbin Laleoglu, Juhana Siljander\\ and Jos\'e Miguel Urbano}

\begin{document}

\begin{abstract}

We complete the study of the regularity for Trudinger's equation by proving that weak solutions are H\"older continuous also in the singular case. The setting is that of a measure space with a doubling non-trivial Borel measure supporting a Poincar\'e inequality. The proof uses the Harnack inequality and intrinsic scaling.

\end{abstract}

\date{\today}



\thanks{Research of JMU supported by CMUC/FCT and projects UTAustin/MAT/0035/2008 and PTDC/MAT/098060/2008.}

\maketitle

\section{Introduction}

The fine properties of non-negative weak solutions of the singular pde
\begin{equation}\label{equation}
\frac{\partial (u^{p-1})}{\partial t}-\div{(|\nabla u|^{p-2}\nabla u)}=0, \quad 1< p \leq 2,
\end{equation}
often referred to as Trudinger's equation, following his pioneering work in \cite{Trud68}, are the object of the present paper. The setting is that of a measure space with a doubling non-trivial Borel measure supporting a Poincar\'e inequality and comprises one of the main novelties in our approach. We show that weak solutions are locally H\"older continuous, completing the effort initiated in \cite{KinnKuus07} and continued recently in~\cite{KuusSiljUrba10}, with the analysis of the degenerate case $p \ge 2$. Although there are some similarities with that case, the new material is substantial and the techniques employed quite different, which should come as no surprise to the expert in the field. The novice might find it appropriate and enlightening to first understand in~\cite{KuusSiljUrba10} how the proof runs for $p \ge 2$.

We have included the case $p=2$ (that corresponds to the heat equation) in the range of variation of $p$ to stress the fact that our proof is stable as $p \rightarrow 2^-$. This stability is achieved through a careful analysis of the dependence of the different constants that pop up along the proof of the H\"older continuity of solutions. The result, although natural and to be expected, was hitherto absent from the literature and is interesting in its own right; it emphasizes the idea that our self-contained proof is the \emph{right} one. 

The issues touched in here have been considered by Porzio and Vespri, Vespri and Ivanov, see~\cite{PorzVesp93, Vesp92, Ivan94}, for equations that are formally equivalent to~\eqref{equation} although it remains unclear whether they possess the same solutions. The relevance of studying equation \eqref{equation} in a measure space, apart from the purely analytical issues it spikes, is related to Harnack inequalities. Grigor'yan and Saloff-Coste observed that the doubling property and the Poincar\'e inequality are not only sufficient but also necessary conditions for the validity of a scale and location invariant parabolic Harnack principle for the heat equation on Riemannian manifolds, see~\cite{Grig91, SaCo02} and~\cite{SaCo92}. Since \eqref{equation} seems to be the only nonlinear parabolic equation which admits such a Harnack inequality, it is the natural candidate for generalizing these results to nonlinear equations.

The analysis of the H\"older continuity of weak solutions is divided, as in the degenerate setting, into two cases: for large scales, when the infimum is much smaller than the oscillation, we can use the Harnack inequality of \cite{KinnKuus07} and the behaviour of solutions is similar to that of caloric functions. For small scales, the equation becomes essentially the singular $p$-Laplace equation treated in \cite{DiBe93,DiBeUrbaVesp04}. There are, nevertheless, many aspects that are new and far from trivial extensions of the known results, namely because of the general measure setting and the extra nonlinear term in the equation. The most relevant are the introduction of a new intrinsic scaling and the different choice of certain test functions, always a crucial aspect in the analysis of pdes. The paper is organized as follows: section 2 collects the basic results on metric spaces needed for our purposes and contains also the main result in the paper; section 3 includes the construction of the iteration argument and the building blocks of the proof, namely the energy and logarithmic estimates; sections 4 and 5 contain the analysis of the two cases the main proof is divided into.

It is never easy to balance the clarity of exposition with avoiding the repetition of previously published material, in particular when the subject is very technical. We made an effort to render the paper as self-contained as possible, keeping it at the same time mostly original. When a duplication of arguments was considered essential for the exposition, this has been expressly acknowledged.

\section{Basics on metric spaces and main result } \label{section:preliminaries}

Let $\mu$ be a Borel measure and $\Omega$ be an open set in $\Rn$.
The Sobolev space $H^{1,p}(\Omega;\mu)$ is defined to be the completion of $C^\infty(\Omega)$ with respect to the Sobolev norm
\[
\|u\|_{1,p,\Omega}=\left(\int_\Omega(|u|^p + |\nabla u|^p)\, d\mu\right)^{1/p}.
\]
A function $u$ belongs to the local Sobolev space $H_{loc}^{1,p}(\Omega;\mu)$ if it belongs to
$H^{1,p}(\Omega';\mu)$ for every $\Omega' \Subset \Omega$. Moreover, the Sobolev space with zero boundary values $H^{1,p}_0(\Omega;\mu)$ is defined as the completion of $C_0^\infty(\Omega)$ with respect to the Sobolev norm. For more properties of Sobolev spaces, see e.g. \cite{HeinKilpMart93}.

Let $t_1<t_2$.
The parabolic Sobolev space $L^p(t_1,t_2;H^{1,p}(\Omega;\mu))$ is the space of functions $u(x,t)$
such that, for almost every $t$, with $t_1<t<t_2$, the function $u(\cdot, t)$ belongs to $H^{1,p}(\Omega;\mu)$ and
\[
\int_{t_1}^{t_2}\int_\Omega (|u|^p+|\nabla u|^p) \, d\nu < \infty,
\]
where we denote $d\nu=d\mu\,dt$.

The definition of the space $L_{loc}^p(t_1,t_2;H_{loc}^{1,p}(\Omega;\mu))$ is analogous.

\begin{definition}
A function $u \in L_{loc}^p(t_1,t_2;H_{loc}^{1,p}(\Omega;\mu))$ is a
weak solution of equation \eqref{equation} in $\Omega\times(t_1,t_2)$ if it satisfies the integral equality
\begin{equation} \label{weak_solution}
\int_{t_1}^{t_2}\int_{\Omega} \left( |\nabla u|^{p-2}\nabla u\cdot \nabla \eta
-u^{p-1} \frac{\partial\eta}{\partial t} \right)\, d\nu = 0
\end{equation}
for every  $\eta \in C_0^\infty(\Omega \times (t_1,t_2))$.
\end{definition}

Next, we recall a few definitions and results from analysis on metric measure spaces.
The measure $\mu$ is doubling if there is a universal constant $D_0\ge 1$ such that
\[
\mu(B(x,2r))\le D_0 \mu(B(x,r)),
\]
for every $B(x,2r)\subset\Omega$. Here
\[
B(x,r):= \left\{ y\in \mathbb{R}^d \, :\, \norm{y-x}<r\right\}
\]
denotes the standard open ball in $\Rn$ with radius $r$ and center $x$. In fact, the Euclidian metric does not play a crucial role in the arguments
and could very well be replaced with another metric. The important feature is the behavior of the measure with respect to the metric.

The dimension related to the doubling measure is $d_\mu=\log_2 D_0$. Observe that for the Lebesgue measure this is $d_\mathcal{L}=d$.
Let $0<r<R<\infty$.
A simple iteration of the doubling condition implies that
$$
\frac{\mu(B(x,R))}{\mu(B(x,r))}
\le C\left(\frac Rr\right)^{d_\mu}.
$$

We further assume that the measure $\mu$ satisfies the following annular decay property.
There exist constants $0<\alpha<1$ and $c\ge 1$ such that
\begin{equation}\label{annular_decay}
\mu(B(x,r)\setminus B(x,(1-\delta)r))\le c\delta^\alpha\mu(B(x,r)),
\end{equation}
for all $B(x,r)\subset\Omega$ and $0<\delta<1$, see~\cite{Buck99}.

The measure is said to support a weak $(q,p)$-Poincar\'e inequality if there exist constants $P_0>0$ and $\tau\ge 1$ such that
\begin{equation}\label{poincare}
\l(\dashint_{B(x,r)}|u-u_{B(x,r)}|^q \, d\mu  \r)^{1/q}
\le P_0 r\left(\dashint_{B(x,\tau r)} |\nabla u|^p \, d\mu\right)^{1/p},
\end{equation}
for every $u\in H^{1,p}_{loc}(\Omega;\mu)$ and $B(x,\tau r)\subset\Omega$.
Here, we denote
\[
u_{B(x,r)}=\dashint_{B(x,r)} u \, d\mu = \frac{1}{\mu(B(x,r))}\int_{B(x,r)} u\, d\mu.
\]
The word weak refers to the constant $\tau$, that may be strictly greater than one.
In $\Rn$ with a doubling measure, the weak $(q,p)$-Poincar\'e inequality
with some $\tau\ge1$ implies the $(q,p)$-Poincar\'e inequality with $\tau=1$,
see Theorem 3.4 in \cite{HajlKosk00}.
Hence, we may assume that $\tau=1$.

We will assume that the measure supports a weak $(1,p)$-Poincar\'e inequality. On the other hand, the weak $(1,p)$-Poincar\'e inequality and the doubling condition
imply a weak $(\kappa,p)$-Sobolev-Poincar\'e inequality with
\begin{equation} \label{kappa}
\kappa =
\begin{cases}
\dfrac{d_\mu p}{d_\mu -p}, & 1<p< d_\mu, \\
2p, & p \ge d_\mu,
\end{cases}
\end{equation}
where $d_\mu$ is as above. For the proof, we refer to \cite{HajlKosk00}.

For Sobolev functions with zero boundary values,
we have the following version of Sobolev's inequality.
Suppose that $u\in  H_0^{1,p}(B(x,r);\mu)$.
Then
\begin{equation}\label{Sobolevzero}
\left(\dashint_{B(x,r)} |u|^\kappa\, d \mu
\right)^{1/\kappa} \le C r \left(\dashint_{B(x,r)}
|\nabla u|^p\, d \mu \right)^{1/p}.
\end{equation}
For the proof we refer, for example, to \cite{KinnShan01}. Moreover, by a recent result in \cite{KeitZhon08},
the weak $(1,p)$-Poincar\'e inequality and the doubling condition also imply
the $(1,q)$-Poincar\'e inequality for some $q<p$. Consequently, we also have the weak $(\kappa,q)$-Sobolev-Poincar\'e inequality which implies, by H\"older's inequality, the $(q,q)$-Poincar\'e inequality for some $q<p$.

In the sequel, we shall refer to \textit{data} as the set of a priori constants
$p$, $d$, $D_0$, and $P_0$.

We can also include a certain kind of test functions in the Poincar\'e inequality.
We recall the following theorem from~\cite{SaCo92}.

\begin{theorem}
Suppose $u \in H^{1,p}(B(x,r); \mu)$ and let
\begin{equation}\label{poincare_phi}
\eta(y)=\min\l\{\frac{2\cdot d(y,\Omega\setminus B(x,r))}{r}, 1\r\}.
\end{equation}
Then there exists a constant $C=C(p,D_0,P_0)>0$ such that
\[
\l(\dashint_{B(x,r)}|u-u_\eta|^p\eta^p \, d\mu  \r)^{1/p}
\le C r\l(\dashint_{B(x,r)} |\nabla u|^p\eta^p \, d\mu\r)^{1/p},
\]
where
\[
u_\eta=\frac{\int_{B(x,r)}u\eta^p\, d\mu}{\int_{B(x,r)}\eta^p \, d\mu}.
\]
\begin{proof}
See Theorem~5.3.4. in~\cite{SaCo92}.
\end{proof}
\end{theorem}
We use this theorem in the form of the following corollary.

\begin{corollary}\label{poincare_corollary}
Suppose $u \in H^{1,p}(B(x,r); \mu)$ and let $\eta$ be as in~\eqref{poincare_phi}. Assume further that
\[
\int_A \eta^p \, d\mu\le \gamma \int_{B(x,r)} \eta^p \, d\mu,
\]
for some $0<\gamma<1$, where $A=\{y\in B(x,r) \ : \ |u(y)|>0\}$. Then there exists a constant $C=C(p, D_0,P_0, \gamma)>0$ such that
\[
\l(\dashint_{B(x,r)} |u|^p\eta^p \, d\mu\r)^{1/p} \le Cr\l(\dashint_{B(x,r)} |\nabla u|^p \eta^p \, d\mu\r)^{1/p}.
\]
\begin{proof}
By the Minkowski inequality and the previous lemma, we get
\begin{equation}\label{Minkowski_estimate}
\begin{split}
&\l(\dashint_{B} |u|^p\eta^p \, d\mu\r)^{1/p} \\
&\qquad \le \l(\dashint_{B} |u-u_\eta|^p\eta^p \, d\mu\r)^{1/p}+\l(\dashint_{B} |u_\eta|^p\eta^p \, d\mu\r)^{1/p} \\
&\qquad\le Cr\l(\dashint_{B} |\nabla u|^p \eta^p \, d\mu\r)^{1/p}+\frac{\int_{B} |u|\eta^p \, d\mu}{\mu(B)^{1/p}(\int_{B}\eta^p \, d\mu)^{1-1/p}}.
\end{split}
\end{equation}
Here we have denoted $B\equiv B(x,r)$.
Using H\"older's inequality, we estimate further
\begin{align*}
\frac{1}{\mu(B)^{1/p}}\int_{B} |u|\eta^p \, d\mu\le \l(\int_{A}\eta^p\, d\mu\r)^{1-1/p}\l(\dashint_{B} |u|^p\eta^p \, d\mu\r)^{1/p}.
\end{align*}
On the other hand, we have
\begin{align*}
\frac{\int_{A}\eta^p\, d\mu}{\int_{B}\eta^p \, d\mu}\le \gamma<1.
\end{align*}
Inserting these estimates in~\eqref{Minkowski_estimate} and absorbing the resulting term to the left hand side finishes the proof.
\end{proof}
\end{corollary}

We will also need the following general result.
\begin{lemma}\label{weak_parabolic}
Let $v$ be a weak solution of equation
\[
v_t-\divt \mathcal{A}(x,t,v,\nabla v) =0,
\]
where the Carath\'eodory function $\mathcal{A}$ satisfies the structure conditions
\begin{align}
&\mathcal{A}(x,t,v,\eta)\cdot \eta \geq \mathcal{A}_0 \norm{\eta}^p, \label{structure_1}\\
&\mathcal{A}(x,t,v,\eta)\leq \mathcal{A}_1 \norm{\eta}^{p-1}, \label{structure_2}
\end{align}
for almost every $(x,t) \in \R^{n} \times \R$ and every $(v,\eta) \in \R \times \R^n$, for
some constants $\mathcal{A}_0,\mathcal{A}_1>0$. Then, for all $k\in \mathbb{R}$, the truncated functions $(v-k)_-$ are weak subsolutions of the same equation with $\mathcal{A}(x,t,v,\nabla v)$ replaced by $-\mathcal{A}(x,t,k-(v-k)_-,-\nabla (v-k)_-)$.
\begin{proof}
See Lemma 1.1 in~\cite{DiBe93}.
\end{proof}
\end{lemma}

Our main result is the following theorem. It still holds for equations with a more general principal part $ \mathcal{A}$, satisfying the standard structure assumptions \eqref{structure_1} and \eqref{structure_2}.

\begin{theorem}\label{main_theorem}
Let $1< p \leq 2$ and assume that the measure $\mu$ is doubling, supports a weak $(1,p)$-Poincar\'e inequality and satisfies the annular decay property. Then
any non-negative weak solution of equation~\eqref{equation} is locally H\"older continuous.
\end{theorem}

Let $(x_{0},t_{0})$ be a point in the space-time domain. The cylinder of radius $r>0$ and height $s>0$, with vertex at $(x_{0},t_{0})$, is defined as
\begin{align*}
Q_{x_{0},t_{0}}(s,r)&:= B(x_0,r) \times (t_{0}-s,t_{0}). 
\end{align*}
We write $Q\left(s,r\right)$ to denote $Q_{0,0}(s,r)$. 
Moreover, we shall use the notation
$\delta Q_{x_{0},t_{0}}(s,r) = Q_{x_{0},t_{0}}(\delta^p s,\delta r)$ and $\delta B(x_0,r) = B(x_0,\delta r)$.

Recall Harnack's inequality from \cite{KinnKuus07}.

\begin{theorem}\label{Harnack}
Let $1<p<\infty$ and suppose that the measure $\mu$ is doubling and supports a weak $(1,p)$-Poincar\'e inequality.
Moreover, let $u\ge 0$ be a weak solution to ~\eqref{equation} in $Q_{x,t}(2r^p,2r)$.
Then, for any $\sigma\in (0,1)$, there exists a constant $ H = H (p,d, D_0, P_0, \sigma) \ge 2$ such that
\begin{equation}\label{eq:Harnack}
\esssup_{B(x,r)\times(t-r^p,t - \sigma r^p) }{u} \leq
 H  \essinf_{B(x,r)\times \left(t - (\sigma/2) r^p, t \right)}{u}.
\end{equation}
\end{theorem}

In particular, the result shows that the weak solutions of equation~\eqref{equation} are locally bounded. In the sequel, we will assume this without further comments.

\section{Iteration argument and fundamental estimates}

Let $u$ be a non-negative weak solution of equation \eqref{equation} in $\Omega\times(t_1,t_2)$ and let $K \subset \Omega\times(t_1,t_2)$. Our aim is to show that the oscillation of $u$ around any point in $K$ reduces in a quantitative way as we suitably decrease the size of the set where the oscillation is measured. This reduction process has to be iterated producing a sequence of cylinders. 

For all purposes in the sequel, it is enough to study the cylinder 
$$Q^0:=Q(R^p,R), \quad R>0,$$ 
instead of the set $K$. Indeed, we could build the sequence of suitable cylinders for any point in $K$ but, since we can always translate the equation, we can, without loss of generality, restrict the study to the origin.

Let
\[
\mu_0^-\le\essinf_{Q^0} u \qquad \text{and} \qquad \mu_0^+\ge\esssup_{Q^0} u,
\]
and define
\[
\omega_0:=\mu_0^+-\mu_0^-.
\]
We may assume that $\omega_0>0$, because otherwise there is nothing to prove.
Furthermore, we choose $\mu_0^-$ small enough so that
\begin{equation}\label{bounds_inf}
(2 H +1)\mu_0^-\le \omega_0
\end{equation}
holds, where $H$ is the constant from Theorem \ref{Harnack}. We will construct a nondecreasing sequence $\{\mu_i^-\}$ and a nonincreasing sequence $\{\mu_i^+\}$ such that
\[
\omega_i :=\mu_i^+-\mu_i^-=\sigma^i\omega_0, \quad i=0,1,\ldots
\]
for some $3/4 < \sigma < 1$. Moreover, these sequences can be chosen so that
\begin{equation}\label{eq:u between mu(i)pm}
\mu_i^-\le\essinf_{Q^i} u \qquad \text{and} \qquad \mu_i^+\geq \esssup_{Q^i}u,
\end{equation}
for some suitable sequence $\{Q^i\}$ of cylinders. Consequently,
\[
\essosc_{Q^i}u\le \omega_i.
\]
The cylinders here will be chosen so that their size decreases in a controllable way from which we will deduce the H\"older continuity. The actual proof will proceed inductively. We will assume that in $Q^i$ the two-sided bound~\eqref{eq:u between mu(i)pm} holds. Then we shall verify an alternative (Cases I and II below), and construct the cylinder $Q^{i+1}$ accordingly, in such a way that~\eqref{eq:u between mu(i)pm} holds with $i$ replaced with $i+1$.

For $\delta >0$, sufficiently small, let 
\begin{equation}\label{sol}
R_i = \delta^i R, \quad i=0,1,\ldots .
\end{equation}
The construction goes as follows: 

\begin{enumerate}

\item for each $Q^i$, choose the numbers $\mu_i^\pm$ so that \eqref{eq:u between mu(i)pm} holds;

\item if $\mu_i^\pm$ satisfy
\begin{equation}\label{eq:Case I test}
\textrm{Case I:} \qquad (2 H +1)\mu_{i}^- \leq \omega_{i} = \mu_i^+ - \mu_i^-,
\end{equation}
then set inductively
\[
Q^{i+1} = Q(R_{i+1}^p,R_{i+1}) =  Q((\delta R_{i})^p,\delta R_{i}) =: \delta Q^i ;
\]

\item if, on the other hand,~\eqref{eq:Case I test} fails for some $i_0>0$ (note that $i_0$ must be larger than zero by~\eqref{bounds_inf}), {\it i.e.}, if
\begin{equation}\label{cor:Case II test}
\textrm{Case II:} \quad (2 H +1)\mu_{i_0}^- > \omega_{i_0}  \ \Longleftrightarrow \ \mu_{i_0}^+ < (2 H +2)\mu_{i_0}^-,
\end{equation}
then set inductively, for all $i\geq i_0$,
\[
Q^{i+1} = Q(R_{i+1}^p,c_{i}  R_{i+1}), 
\quad c_{i} = const \cdot \left(\frac{\omega_{i}}{\mu_{i}^-} \right)^{(p-2)/p}.
\]
\end{enumerate}

By construction, if $\delta \in (0,1/2)$ is small enough, then $Q^{i+1} \subset Q^i$, a fact we will later prove in detail. Observe also that the above sequences have been chosen so that
$$
\mu_i^+ \leq \mu_{i_0}^+ < (2 H +2)\mu_{i_0}^- \leq (2 H +2)\mu_i^-,
$$
for all $i\ge i_0$. 

As mentioned before, Case I reflects the behavior of the equation when the oscillation is large compared to the infimum of $u$. In some sense, the equation resembles the standard heat equation in this case. Case II, on the other hand, deals with the situation in which the equation behaves like an evolutionary $p$-Laplace equation with coefficients.

\subsection{Energy estimates}
We start the proof of Theorem~\ref{main_theorem} by proving the usual energy estimates in a slightly modified setting,
which overcomes the problem that we cannot add constants to solutions, see \cite{DiBeFrie85a,Ivan94} and \cite{Zhou94}.

We introduce the auxiliary function
\begin{align*}
\mathcal{J}((u{-}k)_\pm)= & \pm \int_{k^{p-1}}^{u^{p-1}} \left(\xi^{1/(p-1)} {-} k \right)_\pm \, d\xi 
\\ = & \pm (p-1)\int_{k}^{u} \left(\xi{-}k \right)_\pm \xi^{p-2} \, d\xi 
\\ = & (p-1)\int_0^{(u{-}k)_\pm}(k\pm \xi)^{p-2}\xi \, d\xi
\end{align*}
for which
\begin{equation}\label{J_derivative}
\frac{\partial }{\partial t}\mathcal{J}((u{-}k)_\pm)
=\pm\frac{\partial (u^{p-1})}{\partial t}(u{-}k)_\pm.
\end{equation}
In the sequel, we will need the following estimates. For $1<p \leq 2$, we have
\begin{equation}\label{plus_upper_estimate}
\begin{split}
\mathcal{J}((u{-}k)_+) &= (p-1)\int_0^{(u{-}k)_+}(k+ \xi)^{p-2}\xi \, d\xi\\
&\le (p-1)k^{p-2}\int_0^{(u-k)_+} \xi d\xi \\
&\le (p-1)k^{p-2}\frac{(u{-}k)_+^2}{2}
\end{split}
\end{equation}
and
\begin{equation}\label{plus_lower_estimate}
\begin{split}
\mathcal{J}((u{-}k)_+)
&\ge (p-1)u^{p-2}\int_0^{(u{-}k)_+}\xi\, d\xi \\
&\ge (p-1)u^{p-2}\frac{(u{-}k)_+^2}{2}.
\end{split}
\end{equation}

On the other hand, we also obtain
\begin{equation}\label{minus_upper_estimate}
\begin{split}
\mathcal{J}((u{-}k)_-) &=(p-1)\int_0^{(u{-}k)_-}(k-\xi)^{p-2}\xi \, d\xi\\
&\le (p-1)u^{p-2}\int_0^{(u{-}k)_-}\xi\, d\xi \\
&\le (p-1)u^{p-2}\frac{(u{-}k)_-^2}{2}
\end{split}
\end{equation}
and
\begin{equation}\label{minus_lower_estimate}
\begin{split}
\mathcal{J}((u{-}k)_-)&\ge (p-1)k^{p-2}\int_0^{(u{-}k)_-}\xi\, d\xi \\
&\ge (p-1)k^{p-2}\frac{(u{-}k)_-^2}{2}.
\end{split}
\end{equation}

We are now ready for the fundamental energy estimate. The details of the proof can be found in~\cite{KuusSiljUrba10}.

\begin{lemma}\label{energy}
Let $u\ge 0$ be a weak solution of \eqref{equation} and let $k\ge 0$.
Then there exists a constant $C=C(p)>0$ such that
\begin{equation*}
\begin{split}
&\esssup_{t_1<t<t_2}\int_{\Omega}\mathcal{J}((u{-}k)_\pm)\varphi^p\, d\mu
+\int_{t_1}^{t_2}\int_{\Omega} |\nabla (u{-}k)_\pm\varphi|^{p}
\, d\nu \\
&\le C\int_{t_1}^{t_2}\int_{\Omega}(u{-}k)_\pm^p|\nabla \varphi|^p\, d\nu
+C\int_{t_1}^{t_2}\int_{\Omega} \mathcal{J}((u{-}k)_\pm)\varphi^{p-1}\left(\frac{\partial \varphi}{\partial t}\right)_+\, d\nu , \\
\end{split}
\end{equation*}
for every nonnegative $\varphi\in C_0^\infty(\Omega\times(t_1,t_2))$.
\end{lemma}

\begin{proof}
Choose the test function $\eta=\pm(u{-}k)_{\pm}\varphi^p$ in~\eqref{weak_solution}, modulo suitable regularization in the time direction, and use~\eqref{J_derivative} for the parabolic term. The rest of the proof is standard, see~\cite{KuusSiljUrba10}.
\end{proof}

By the formal substitution $v=u^{p-1}$, equation~\eqref{equation} transforms into
\begin{equation}\label{substituted_equation}
v_t-c\divt(v^{2-p}|\nabla v|^{p-2}\nabla v)=0,
\end{equation}
where $c=1/(p-1)^{p-1}$. In the second part of the paper, where we prove the reduction of the oscillation for Case II, we are going to use this transformed equation instead of the original one.

The weak solutions of \eqref{substituted_equation} are defined in a similar fashion as for equation~\eqref{equation}. However, it is not clear whether the weak solutions of the two equations are the same in general. Nevertheless, under appropriate assumptions, the substitution can be justified; that is the contents of the following lemma.

\begin{lemma}
Let $0<\Lambda^-\le u\le \Lambda^+$ and $v=u^{p-1}, 1<p<\infty$. Then $u$ is a weak solution of equation~\eqref{equation} if, and only if, $v$ is a weak solution of equation~\eqref{substituted_equation}.

\begin{proof}
Since $f(x)=x^{p-1}$, $p>1$, is Lipschitz in $[\Lambda^-,\Lambda^+]$, by the chain rule we have
\begin{equation}\label{gradient_equivalence}
\nabla u=\frac{1}{p-1} v^{(2-p)/(p-1)}\nabla v.
\end{equation}
Thus
\[
|\nabla u|^{p-2}\nabla u = \frac{1}{(p-1)^{p-1}}v^{2-p}|\nabla v|^{p-2}\nabla v.
\]
This implies that formally the equations are equivalent. It remains to show that the function spaces where the solutions are defined are the same. But now, since $v$ is bounded above and below by a positive number, by~\eqref{gradient_equivalence} we have $\nabla u \in L^p$ if and only if $\nabla v \in L^p$. This concludes the proof.
\end{proof}
\end{lemma}
We have the following energy estimate for equation~\eqref{substituted_equation}.

\begin{lemma}\label{substituted_energy}
Let $0<\Upsilon^-\le v\le \Upsilon ^+$ be a weak solution of \eqref{substituted_equation} and let $k\ge 0$.
Then there exists a constant $C=C(p)>0$ such that
\begin{equation*}
\begin{split}
&(\Upsilon ^+)^{p-2}\esssup_{t_1<t<t_2}\int_{\Omega}(v-k)_\pm^2\varphi^p\, d\mu
+\int_{t_1}^{t_2}\int_{\Omega} |\nabla (v-k)_\pm\varphi|^{p}
\, d\nu \\
&\le C\int_{t_1}^{t_2}\int_{\Omega}(v-k)_\pm^p|\nabla \varphi|^p\, d\nu
+C(\Upsilon ^-)^{p-2}\int_{t_1}^{t_2}\int_{\Omega} (v-k)_\pm^2\varphi^{p-1}\left(\frac{\partial \varphi}{\partial t}\right)_+\, d\nu , \\
\end{split}
\end{equation*}
for every nonnegative $\varphi\in C_0^\infty(\Omega\times(t_1,t_2))$.
\begin{proof}
The claim follows by choosing as test function $\eta=\pm(v-k)_\pm\varphi^p$ in the definition of weak solution and by estimating $v^{2-p}$ by $(\Upsilon ^-)^{2-p}$ and $(\Upsilon ^+)^{2-p}$. Otherwise, the proof is standard and the details are left to the reader.
\end{proof}
\end{lemma}

We denote
\[
\psi_\pm(v):=\Psi(H_k^\pm, (v{-}k)_\pm, c)=\l(\ln\l(\frac{H_k^\pm}{c+H_k^\pm-(v{-}k)_\pm}\r)\r)_+.
\]
The following logarithmic lemma will be used to forward information in time. The proof is somehow formal since it involves the time derivative of $v$. A rigorous justification requires the use of Steklov averages in the spirit of ~\cite[p. 101--102]{DiBe93} (see also \cite{KuusSiljUrba10}).

\begin{lemma}\label{Harnack_logarithm}
Let $v\ge 0$ be a weak solution of equation~\eqref{substituted_equation}. Then there exists a constant $C=C(p)>0$ such that, for $1<p \leq 2$, we have
\begin{align*}
\esssup_{t_1<t<t_2}&\int_{\Omega} \psi_+^2(v)(x,t)\varphi^p(x) \, d\mu \\
&\le \int_{\Omega} \psi_\pm^2(v)(x,t_1)\varphi^p(x) \, d\mu \\
&+C \int_{t_1}^{t_2}\int_\Omega v^{2-p}\psi_\pm(v)|(\psi_\pm)^{'}(v)|^{2-p}|\nabla  \varphi|^p \, d\nu.
\end{align*}
Above, $\varphi \in C_0^\infty(\Omega)$ is any nonnegative time-independent test function.

\begin{proof}
Choose
\[
\eta_\pm(v)=\frac{\partial}{\partial v}(\psi_\pm^2(v))\varphi^p
\]
in the definition of weak solution and observe that
\begin{equation}\label{psi''}
(\psi_\pm^2)''=2(1+\psi_\pm)(\psi_\pm')^2.
\end{equation}

The parabolic term will take the form
\begin{align*}
&\int_{t_1}^{t_2}\int_\Omega \frac{\partial v}{\partial t}  \eta_\pm(v)\, d\nu= \int_{t_1}^{t_2}\int_\Omega \frac{\partial v}{\partial t} \frac{\partial}{\partial v}(\psi_\pm^2(v))\varphi^p \, d\nu\\
&=\int_{t_1}^{t_2}\int_\Omega \frac{\partial}{\partial t}(\psi_\pm^2(v))\varphi^p \, d\nu\\
&=\l[ \int_\Omega \psi_\pm^2(v)\varphi^p d\mu\r]_{t_1}^{t_2}.
\end{align*}
Here we used the fact that $\varphi$ is time-independent.

On the other hand, by using~\eqref{psi''} together with Young's inequality, we obtain
\begin{align*}
&  v^{2-p}|\nabla v|^{p-2}\nabla v \cdot \nabla \eta_\pm
 =  v^{2-p}|\nabla v|^{p-2}\nabla v \cdot \nabla((\psi_\pm^2(v))'\varphi^p)
\\ & \qquad \geq v^{2-p}\big(2 |\nabla v|^{p}(1+\psi_\pm)(\psi_\pm')^2\varphi^p
-2p|\nabla v|^{p-1}  \psi_\pm\psi_\pm'\varphi^{p-1}|\nabla\varphi|\big)
\\ & \qquad \geq  v^{2-p}\l( |\nabla v|^{p}(\psi_\pm')^2\varphi^p
- C\psi_\pm(\psi_\pm')^{2-p}|\nabla\varphi|^p\r),
\end{align*}
almost everywhere, from which the claim follows.
\end{proof}
\end{lemma}

We will use the following notation in the next lemma, which is one of the most crucial parts of the argument. Let
\[
r_n=\frac{r}{2}+\frac{r}{2^{n+1}},
\qquad Q_{n}^\pm=B_{n}\times T_{n}^\pm=B(\bar x, \gamma_1^\pm r_n)\times (t^*-\gamma_2^\pm r_n^p,t^*)
\]
for $n=0,1,\dots$, where $\gamma_1^\pm$ and $\gamma_2^\pm$ are constants.

\begin{lemma}\label{main_lemma}
Suppose $u\ge0$ is a weak solution of equation~\eqref{equation} in $Q^i$, and let $\mu_i^\pm$ be positive numbers such that
\[
\mu_i^- \leq \essinf_{Q^i}{u} \leq \esssup_{Q^i}{u} \leq \mu_i^+, \qquad \mu_i^+-\mu_i^-=:\omega_i.
\]
Let
\[
k_n^\pm = \mu_i^\pm \mp \frac{\eps_\pm \omega_i}{4}\left(1 + \frac1{2^{n}}  \right), \qquad n=0,1,\dots,
\]
for $0\leq \eps^\pm \leq 1$. In the minus case, assume further  that
\begin{equation}\label{odd_assumption} 
\eps_- \omega_i \leq C_0 \mu_i^-
\end{equation}
for all $n=0,1,\dots$, for some $C_0>0$. Then there are positive constants $C_-\equiv C_-(D_0,P_0,C_0,p)$ and $C_+ \equiv C_+(D_0,P_0,p)$ such that
if $Q_0^\pm \subset Q^i$, then
$$
\frac{\nu(A_{n+1}^\pm)}{\nu(Q_{n+1}^\pm)}\le C_\pm^{n+1} \Gamma_n^\pm  \left(  \frac{\nu(A_n^\pm)}{\nu(Q_n^\pm)}\right)^{2-p/\kappa},
$$
for every $n=0,1,\dots$. Here 
\[
A_{n}^\pm=\left\{(x,t)\in Q_{n}^\pm:\pm u(x,t)>\pm k_n^\pm\right\},
\]
$\kappa$ is the Sobolev exponent as in \eqref{kappa}, and
\[
\Gamma_n^\pm = \frac{(\gamma_1^\pm)^p}{\gamma_2^\pm}  \left(\frac{k_n^\pm}{\eps_\pm \omega_i}\right)^{p-2}
\left(\frac{\gamma_2^\pm}{(\gamma_1^\pm)^p} \left(\frac{\eps_\pm \omega_i}{k_n^\pm}\right)^{p-2} +  1 \right)^{2-p/\kappa},
\]
for all $n=0,1,\dots$.

\begin{proof}
Choose cutoff functions $\varphi_n^\pm \in C^\infty(Q_n^\pm)$, vanishing on the parabolic boundary of $Q_n^\pm$, and such that $0\le\varphi_n^\pm\le 1$,
$\varphi_n^\pm=1$ in $Q_{n+1}^\pm$,
\begin{align}\label{gradient_estimates}
|\nabla \varphi_n^\pm|\le\frac{C2^{n+1}}{ \gamma_1^\pm r}
\quad\text{and}\quad \left(\frac{\partial \varphi_n^\pm}{\partial t}\right)_+\le \frac{C2^{p(n+1)}}{\gamma_2^\pm r^p}.
\end{align}
Denote in short
\[
\quad u_n = (u{-}k_n)_\pm, \qquad k_n = k_n^\pm, \qquad \eps = \eps_\pm,\qquad \gamma_1=\gamma_1^\pm,\qquad \gamma_2=\gamma_2^\pm
\]
and
\[
Q_n = B_n \times T_n^\pm = Q_{n}^\pm, \qquad A_n=A_n^\pm, \qquad \varphi_n = \varphi_n^\pm.
\]

By H\"older's inequality, together with the Sobolev inequality \eqref{Sobolevzero}, we obtain
\begin{equation}\label{HoldSobo temp}
\begin{split}
\dashint_{Q_{n+1}} & u_n^{2(1-p/\kappa)+ p} \, d\nu
\\ \leq &
\frac{\nu(Q_{n})}{\nu(Q_{n+1})}
\dashint_{Q_{n}}  u_n^{2(1-p/\kappa)+ p} \varphi_n^{p(1-p/\kappa)+ p} \, d\nu \\
\leq & C \dashint_{T_{n}}\left(\dashint_{B_{n}} u_n^2\varphi_n^p\, d\mu\right)^{1-p/\kappa}
\left(\dashint_{B_{n}}(u_n \varphi_n)^{\kappa}\, d\mu\right)^{p/\kappa}\, dt \\
\leq &
C (\gamma_1^\pm)^pr_{n}^p  \left(\esssup_{T_n} \dashint_{B_{n}} u_n^2\varphi_n^p \, d\mu\right)^{1-p/\kappa}
\dashint_{Q_n} |\nabla (u_n\varphi_n)|^p\,d\nu.
\end{split}
\end{equation}
Here, we applied the doubling property of the measure $\nu$, giving
\[
\frac{\nu(Q_{n})}{\nu(Q_{n+1})} \leq C.
\]

We continue by studying the term involving the essential supremum. Since $\eps_+ \leq 1$, we have
\[
\mu_i^+ = k_{n}^+ + \frac{\eps_+ \omega_i}{4}\left(1 + \frac1{2^{n}}  \right) \leq k_n^+ + \frac{\mu_i^+}{2} \quad \Longrightarrow \quad \mu_i^+ \leq 2k_n^+,
\]
and hence we obtain, using \eqref{plus_lower_estimate},
\begin{align*}
(u-k_n^+)_+^{2} \leq  \frac2{p-1} u^{2-p} \mathcal{J}((u-k_n^+)_+) \le C (k_n^+)^{2-p}\mathcal{J}((u-k_n^+)_+).
\end{align*}
On the other hand,
the lower bound~\eqref{minus_lower_estimate} immediately gives
\begin{align*}
(u-k_n^-)_-^2 &\leq \frac2{p-1} (k_n^-)^{2-p} \mathcal{J}((u-k_n^-)_-).
\end{align*}
Using these estimates together with the energy estimate of Lemma~\ref{energy}, yields
\begin{align*}
\esssup_{T_n} & \dashint_{B_{n}} u_n^2\varphi_n^p \, d\mu
\leq
C (k_n)^{2-p}
\esssup_{T_n} \dashint_{B_{n}} \mathcal{J}(u_n)\varphi_n^p\, d\mu
\\ \leq &
C (k_n)^{2-p}  \gamma_2 r_n^p \dashint_{Q_n} \left(  u_n^p |\nabla \varphi_n|^p
+  \mathcal{J}(u_n) \varphi_n^{p-1}\left(\frac{\partial \varphi_n}{\partial t}\right)_+  \right) \,d\nu.
\end{align*}
Furthermore, estimates~\eqref{plus_upper_estimate} and~\eqref{minus_upper_estimate}, together with~\eqref{odd_assumption}, imply
\[
\mathcal{J}((u{-}k_n^+)_+) \leq C (k_n^+)^{p-2}(u{-}k_n^+)_+^2 ,\quad  \mathcal{J}((u{-}k_n^-)_-) \leq C (k_n^-)^{p-2}(u{-}k_n^-)_-^2.
\]

Next, using~\eqref{gradient_estimates}, we arrive at
\begin{equation}\label{eq:main lemma estimate}
\begin{split}
& r_n^p \dashint_{Q_n} \left(  u_n^p |\nabla \varphi_n|^p
+  \mathcal{J}(u_n) \varphi_n^{p-1}\left(\frac{\partial \varphi_n}{\partial t}\right)_+  \right) \,d\nu  \\
& \qquad  \leq
C  2^{np}\dashint_{Q_n} \left(   \frac{u_n^p}{\gamma_1^p} + \frac{(k_n)^{p-2}}{\gamma_2} u_n^2  \right)\, d\nu \\
& \qquad \leq C 2^{np} (\eps \omega_i)^p
\left(\frac{1}{\gamma_1^p} + \frac1{\gamma_2} \left(\frac{\eps \omega_i}{k_n}\right)^{2-p}  \right) \frac{\nu(A_n)}{\nu(Q_n)},
\end{split}
\end{equation}
where the last inequality follows from the fact that $(u{-}k_n)_\pm \leq \eps_\pm \omega_i$ almost everywhere.
Thus, we conclude with
\begin{equation}\label{eq:main lemma 1}
\esssup_{T_n}  \dashint_{B_{n}} u_n^2\varphi_n^p \, d\mu \leq
C 2^{np} (\eps \omega_i)^2
\left(\frac{\gamma_2}{\gamma_1^p} \left(\frac{\eps \omega_i}{k_n}\right)^{p-2} +  1 \right) \frac{\nu(A_n)}{\nu(Q_n)}.
\end{equation}
Furthermore, since
\[
\dashint_{Q_n} |\nabla (u_n\varphi_n)|^p\,d\nu \leq
C\dashint_{Q_n} |\nabla u_n|^p \varphi_n^p\,d\nu + C\dashint_{Q_n} u_n^p |\nabla\varphi_n|^p\,d\nu,
\]
applying again the energy estimate and~\eqref{eq:main lemma estimate} leads to
\begin{equation*}\label{eq:main lemma 2}
\begin{split}
r_{n}^p \gamma_1^p \dashint_{Q_n} |\nabla (u_n\varphi_n)|^p\,d\nu
\leq
C 2^{np}  \left(\eps \omega_i\right)^p
\left(1 + \frac{\gamma_1^p}{\gamma_2} \left(\frac{\eps \omega_i}{k_n}\right)^{2-p}  \right) \frac{\nu(A_n)}{\nu(Q_n)} \nonumber \\
\end{split}
\end{equation*}

\begin{equation}\label{mardigras}
\begin{split}
\qquad \qquad \leq C 2^{np}  \left(\eps \omega_i\right)^p \frac{\gamma_1^p}{\gamma_2}\left(\frac{k_{n}}{\eps \omega_i}\right)^{p-2}
\left( \frac{\gamma_2}{\gamma_1^p} \left(\frac{\eps \omega_i}{k_n}\right)^{p-2} +1 \right) \frac{\nu(A_n)}{\nu(Q_n)}.
\end{split}
\end{equation}

To finish the proof, note first that
\begin{align*}
(u-k_{n}^\pm)_\pm \chi_{\{(u-k_n^\pm)_\pm>0\}}  \geq &
(u-k_{n}^\pm)_\pm \chi_{\{(u-k_{n+1}^\pm)_\pm>0\}}
\\ \geq &
|k_{n+1}^\pm-k_n^\pm|
\\ \geq  & 2^{-(n+3)}\eps_\pm \omega_i.
\end{align*}
It then follows that
\begin{equation}\label{eq:main lemma 3}
\dashint_{Q_{n+1}}  u_n^{2(1-p/\kappa)+ p} \, d\nu  \geq  \left(2^{-(n+3)} \eps \omega_i\right)^{2(1-p/\kappa)+ p}
\frac{\nu(A_{n+1})}{\nu(Q_{n+1})}.
\end{equation}
Inserting estimates~\eqref{eq:main lemma 1}, ~\eqref{mardigras}
and ~\eqref{eq:main lemma 3} into~\eqref{HoldSobo temp} concludes the proof.
\end{proof}
\end{lemma}

We now obtain an important corollary of the previous result, using the following ``lemma on fast geometric convergence'' (cf. \cite{DiBe93}).

\begin{lemma}\label{geometric_convergence}
Let $(Y_n)_n$ be a sequence of positive numbers satisfying
$$
Y_{n+1}\le Cb^nY_n^{1+\alpha},
$$
where $C,b>1$ and $\alpha>0$. Then $(Y_n)_n$ converges to zero as $n\rightarrow\infty$ provided
$$
Y_0\le C^{-1/\alpha}b^{1-\alpha^2}.
$$
\end{lemma}

\begin{corollary} \label{cor:iteration}
Under the assumptions of Lemma~\ref{main_lemma} and the additional condition 
\begin{equation}\label{eq:iteration cor ass}
    \frac1c \leq \frac{(\gamma_1^\pm)^p}{\gamma_2^\pm}  \left(\frac{k_n^\pm}{\eps_\pm \omega_i}\right)^{p-2} \leq c,
\end{equation}
for some constant $c\geq 1$ and for all $n\geq 0$, there exits a constant $\alpha_0$, depending only on the data, $c$, and, in the minus case, also on $C_0$, such that
\begin{equation}\label{eq:iteration cor}
\frac{\nu(A_0^\pm)}{\nu(Q_0^\pm)} \le \alpha_0^\pm 
\quad \Longrightarrow \quad \pm u(x,t) \leq \pm \mu_i^\pm - \frac{\eps_\pm \omega_i}{4},
\end{equation}
for almost every  $(x,t) \in B(\bar x, \gamma_1^\pm r/2)\times (t^*-\gamma_2^\pm (r/2)^p,t^*)$.
\end{corollary}
\begin{proof}
By setting 
\[
Y_n = \frac{\nu(A_n^\pm)}{\nu(Q_n^\pm)},
\]
we find, by~\eqref{eq:iteration cor ass} and Lemma~\ref{main_lemma}, that 
\[
Y_{n+1} \le C_\pm^{n+1} Y_n^{2-p/\kappa},
\] 
for some constant $C_\pm$, depending only on the data, $c$, and, in the minus case, also on $C_0$. Putting 
$$\alpha_0^\pm = C_\pm^{-1/(1-p/\kappa)+1-(1-p/\kappa)^2},$$ 
we conclude, using Lemma~\ref{geometric_convergence}, that if the left-hand side of \eqref{eq:iteration cor} holds then $Y_n \to 0$, as $n\to \infty$; the result follows.
\end{proof}

\begin{remark}\label{main_lemma_remark}
Lemma \ref{main_lemma} and Corollary \ref{cor:iteration} were proved for weak solutions of equation~\eqref{equation}. However, essentially the same proof also works for the weak solutions of equation~\eqref{substituted_equation} since it only uses the energy estimate. Actually, in this case, we do not need to assume~\eqref{odd_assumption} in Lemma~\ref{main_lemma}.
\end{remark}

\section{Analysis of case I}

In this section, we assume that~\eqref{eq:Case I test} holds for an index $i$. Our aim is to show that the measures of certain distribution sets
tend to zero and that the local H\"older continuity follows from this. Most of the
proofs in this section are similar to the corresponding ones in~\cite{KuusSiljUrba10} but we include them here for the sake of completeness.

We start by studying the cylinder $Q^i \equiv Q(R_i^p,R_i)$. 
Since the leading assumption in Case I is $(2 H +1)\mu_{i}^- \leq \omega_{i}$,  we 
have that $\mu_i^- \leq \omega_i$, which implies the second inequality in $\omega_i \leq \mu_i^+ \leq 2\omega_i$ (the first one is trivial). For $\gamma_1^+=1, \gamma_2^+=1$ and $\eps_+=1$, \eqref{eq:iteration cor ass} is then satisfied, for a constant $c$ depending only on $p$, and we choose $\alpha_0^+ \equiv \alpha_0^+(data) \in (0,1)$ from Corollary~\ref{cor:iteration} accordingly.

After fixing $\alpha_0^+ \in (0,1)$, we choose $ H  \equiv  H (\alpha_0^+,data) \equiv  H (data) \geq 2$, according to the
Harnack inequality in Theorem~\ref{Harnack}, such that
\begin{equation}\label{eq:Harnack H(0)}
\esssup_{B(0,r)\times(-r^p,-\alpha_0^+ r^p) }{u} \leq
 H  \essinf_{Q\left((\alpha_0^+ /2)r^p,r \right)}{u}
\end{equation}
for all $0<r\leq R_i$. In fact, strictly speaking, this only holds for $0<r\leq R_i/2$ but since we could have started with the radius $2R$ it does not carry any loss of generality.
This choice fixes the constant $H$ in~\eqref{eq:Case I test} as well; observe carefully that it
is independent of $r$ and $\mu_i^\pm$, or the index $i$ in general.

We will study two different alternatives which are considered in the following two lemmata, respectively.

\begin{lemma} \label{lem:Case I lemma 1}
Let $u\ge 0$ be a weak solution of equation~\eqref{equation} in $Q^i$ satisfying~\eqref{eq:Case I test}.
If
\begin{equation}\label{measure_assumption2}
\nu\l(\brc{(x,t)\in B(0,R_i)\times (-R_i^p,-\alpha_0^+ R_i^p) \ : \ u(x,t)\ge \mu_i^-+\frac{\omega_i}{2}}\r)=0
\end{equation}
then
\[
\essosc_{\frac12Q^i}{u}\le \frac34 \omega_i.
\]
\begin{proof}
With $Q_0^+ = Q^i$, and observing that $\mu_i^-+\frac{\omega_i}{2}= \mu_i^+-\frac{\omega_i}{2}=k_0^+$, assumption~\eqref{measure_assumption2} immediately gives
\[
\frac{\nu(A_0^+)}{\nu(Q_0^+)} \le\alpha_0^+ .
\]
By the choices preceding the statement of the lemma, it follows from Corollary~\ref{cor:iteration} 
that
\[
\esssup_{\frac12 Q^i }{u}
\le \mu_i^+-\omega_i/4.
\]
Since 
\[
\essinf_{\frac12 Q^i }{u}
\geq \essinf_{Q^i }{u} \geq \mu_i^-,
\]
the result follows.
\end{proof}

\end{lemma}

Next, we study the alternative.

\begin{lemma}\label{forwarding}
Let $u\ge 0$ be a weak solution of equation~\eqref{equation} in $Q^i$ satisfying~\eqref{eq:Case I test}.
If
\begin{equation}\label{measure_assumption}
\nu\l(\brc{(x,t)\in B(0,R_i)\times (-R_i^p,-\alpha_0^+ R_i^p):u(x,t)\ge \mu_i^-+\frac{\omega_i}{2}}\r)>0
\end{equation}
then there exists a constant $\sigma_I\equiv\sigma_I(data)\in(3/4,1)$ such that
\[
\essosc_{(\alpha_0^+/2) Q^i}{u}\le \sigma_I\omega_i.
\]
\begin{proof}
By assumption~\eqref{measure_assumption}, we have
\[
\esssup_{B(0,R_i)\times (-R_i^p,-\alpha_0^+ R_i^p)}{u}\ge \mu_i^-+\frac{\omega_i}{2}.
\]
Now we can use Harnack's inequality~\eqref{eq:Harnack H(0)}, together with assumption~\eqref{eq:Case I test}, to deduce
\begin{align*}
\essinf_{(\alpha_0^+/2) Q^i }{u}& \ge \essinf_{Q\left((\alpha_0^+ /2)R_i^p,R_i \right)}{u} \\
&\ge\frac{1}{ H }\esssup_{B(0,R_i)\times (-R_i^p,-\alpha_0^+ R_i^p)}{u} \\
&\ge  \frac{\mu_i^-}{ H }+\frac{\omega_i}{2 H } \\ 
&\ge \mu_i^-+ \frac{\mu_i^-}{ H }-\frac{\omega_i}{2 H +1}+\frac{\omega_i}{2 H } \\
&\ge \mu_i^-+\frac{\omega_i}{2 H (2 H +1)}.
\end{align*}
Since 
\[
\esssup_{(\alpha_0^+/2) Q^i }{u}
\leq \esssup_{Q^i }{u} \leq \mu_i^+,
\]
we obtain
\[
\essosc_{(\alpha_0^+/2) Q^i }{u}
\le \omega_i-\frac{\omega_i}{2 H (2 H +1)}
=\sigma_I\omega_i,
\]
with
\[
\sigma_I=1-\frac{1}{2 H (2 H +1)} \in (3/4,1).
\]
This completes the proof.
\end{proof}
\end{lemma}

As a consequence of the above two lemmata, we obtain the concluding result of this section.
\begin{corollary} \label{cor:Case I conclusion}
Suppose that~\eqref{eq:u between mu(i)pm}
holds in $Q^i=Q(R_i^p,R_i)$ and that~\eqref{eq:Case I test} is satisfied. Then
there are constants $\sigma_I,\delta_I \in (0,1)$,  both depending only on the data, such that
\begin{equation}\label{shrink}
    \essosc_{Q((\delta_I R_i)^p, \delta_I R_i)}{u} \leq \sigma_I \omega_{i}.
\end{equation}
\end{corollary}

We are now in a position of proceeding with the iteration. We put 
$$Q^{i+1}=Q((\delta_I R_i)^p, \delta_I R_i)$$
and observe that, due to \eqref{shrink}, as we go from $Q^{i}$ to $Q^{i+1}$, the oscillation decreases in a quantitative way. Moreover, we can define, in the case of the first alternative,
\[
\mu_{i+1}^+=\mu_i^+-\omega_i/4  \qquad \textrm{and} \qquad \mu_{i+1}^-=\mu_i^-,
\]
and, in the case of the second alternative,
\[
\mu_{i+1}^+=\mu_i^+ \qquad \textrm{and} \qquad \mu_{i+1}^-=\mu_i^-+\frac{\omega_i}{2 H (2 H +1)}
\]
and so~\eqref{eq:u between mu(i)pm} holds with $i$ replaced with $i+1$.

\section{Analysis of case II}

In this section, we will work with equation~\eqref{substituted_equation} and the corresponding solution $v=u^{p-1}$. Due to assumption \eqref{cor:Case II test} of this case, the equation behaves like the evolutionary $p$-Laplace equation. The difficulty is that we can no longer use the Harnack principle, as the lower bound it gives might be trivial. Indeed, the infimum can be larger than the lower bound given by Harnack's inequality.

Suppose that $i_0$ is the first index for which assumption~\eqref{cor:Case II test} holds, immediately
implying {\em an elliptic Harnack estimate} for all $i\geq i_0$:
\begin{equation}\label{Harnack_assumption_u}
\mu_i^+ \leq \mu_{i_0}^+ = \omega_{i_0} +\mu_{i_0}^- < (2 H +2)\mu_{i_0}^-\leq (2 H +2)\mu_{i}^-.
\end{equation}
Observe also that
\begin{align*}
\frac{\omega_{i_0}}{(2 H +1)} & < \mu_{i_0}^- \\
& \leq \mu_{i_0-1}^-+(1-\sigma)\omega_{i_0-1} \\
&\leq\left(\frac1{2H+1} + 1 - \sigma\right)\omega_{i_0-1} \leq \omega_{i_0},
\end{align*}
since $\omega_{i_0}=\sigma\omega_{i_0-1}$, $H\geq 2$, and $\sigma \in (3/4,1)$.
Thus, we obtain
\begin{equation}\label{inclusion_estimate}
1\leq \frac{\omega_{i_0}}{\mu_{i_0}^-} < 2H+1.
\end{equation}

Our goal in Case II is to show that for the function $v\equiv u^{p-1}$ we have
$$
\essinf_{Q{i}}\upsilon \ge (\mu_{i}^-)^{p-1}=:\bar{\mu}_{i}^-,\quad\esssup_{Q{i}}\upsilon \le(\mu_{i}^+)^{p-1}=:\bar{\mu}_{i}^+
$$
and
$$
\essosc_{Q{i}}\upsilon \le\bar{\mu}_{i}^+ - \bar{\mu}_{i}^- 
=
\bar{\omega}_{i} := \sigma^{i-i_0} \bar{\omega}_{i_0}.
$$
for $i=i_0,i_0+1,\ldots$, with some $\sigma \in (0,1)$.
Note also that it follows from~\eqref{Harnack_assumption_u}  that
\begin{equation}\label{eq:bar est}
\bar{\mu}_{i}^+ \leq C_1 \bar{\mu}_{i}^-,
\end{equation}
for some $C_1 \equiv C_1(H) \equiv C_1(data) \geq 1$ and for all $i \geq i_0$.
We shall show that an oscillation reduction for $\upsilon$ implies an oscillation
reduction also for $u$ in a quantitative way.

While analyzing Case II, we shall consider cylinders of the type
\[
Q^{i+1}:=Q\left(R_{i+1}^p,c_i R_{i+1}\right)
, \qquad
c_i:= \left( \frac{\bar{\omega}_{i}}{2^{\lambda} \bar{\mu}_{i}^-}\right)^{\frac{p-2}{p}},
\]
where $\lambda >0$ shall be specified later, depending only on the data.
By~\eqref{inclusion_estimate} and \eqref{Harnack_assumption_u}, we have
\[
c_{i_0} \leq \left( \frac{(\mu_{i_0}^+)^{p-2}  \omega_{i_0} }{2^{\lambda} (\mu_{i_0}^-)^{p-1}}\right)^{\frac{p-2}{p}}
\leq \left( \frac{(\mu_{i_0}^+)^{p-2} }{2^{\lambda} (\mu_{i_0}^-)^{p-2}}\right)^{\frac{p-2}{p}} \leq
\left(C_2 2^{\lambda} \right)^{\frac{2-p}{p}}, 
\]
for some constant $C_2>1$ depending only on the data.
We now set
\[
r_{i_0} := \frac18 \left( C_2 2^{\lambda} \right)^{(p-2)/p} R_{i_0},
\]
{\it i.e.}, shrink the cylinder $Q^{i_0}$ by a factor depending on $\lambda$ and the data.
We therefore have that $8Q\left(r_{i_0}^p,c_{i_0} r_{i_0}\right) \subset Q^{i_0}$ and,
in particular, the essential extremal values and oscillations of both $u$ and $v$ in  $8Q\left(r_{i_0}^p,c_{i_0} r_{i_0}\right)$ are controlled
by the quantities $\mu_{i_0}^\pm$ and $\bar{\mu}_{i_0}^\pm$, respectively.
Furthermore, observing that
\[
\frac{c_{i+1}}{c_{i}} = \left(\frac{\sigma \bar{\mu}_{i}^-}{\bar{\mu}_{i+1}^-} \right)^{\frac{p-2}{p}} \leq
\left(\frac{\sigma \bar{\mu}_{i}^-}{\bar{\mu}_{i}^- + (1-\sigma) \bar{\omega_i}} \right)^{\frac{p-2}{p}} \leq C_3^{\frac{2-p}{p}},
\]
where the constant $C_3>1$ depends only on the data, we have that
\begin{equation}\label{eq:Case II inclusion}
Q(r^p,c_{i+1} r) \subset Q(C_3^{2-p} r^p,c_{i} C_3^{(2-p)/p} r)
\end{equation}
for all $r>0$. Moreover, since $p \leq 2$, we have, applying the mean value theorem to $f(x)=x^{p-1}$ in the interval $[\mu_{i_0}^-,\mu_{i_0}^+]$,
\begin{equation}\label{eq:bar osc connection 12} \nonumber
\bar{\omega}_{i_0} \leq  (\mu_{i_0}^-)^{p-2} \omega_{i_0},
\end{equation}
and therefore
\[
\frac{\bar{\omega}_{i}}{2^{\lambda} \bar{\mu}_{i}^-}  \leq \frac{\bar{\omega}_{i_0}}{ 2^{\lambda} \bar{\mu}_{i_0}^-} \leq 
\frac{\omega_{i_0}}{ 2^{\lambda} \mu_{i_0}^-} \leq \frac{2H+1}{2^\lambda}
\]
holds. Thus, by assuming
\begin{equation} \label{lambda large}
2^\lambda \geq 2H+1,
\end{equation}
we obtain that
\begin{equation} \label{c(i) large}
c_i  = \left( \frac{\bar{\omega}_{i}}{2^{\lambda} \bar{\mu}_{i}^-}\right)^{\frac{p-2}{p}} \geq 1.
\end{equation}
Henceforth, we shall always assume~\eqref{lambda large}.
\\

We are now ready to start the argument in Case II. Assume that in $Q^i$, $i\geq i_0$, we have
$$
\essinf_{Q^{i}}\upsilon \geq \bar{\mu}_{i}^-,\qquad
\esssup_{Q^{i}}\upsilon \leq \bar{\mu}_{i}^+, \qquad \bar{\mu}_{i}^+-\bar{\mu}_{i}^- =\bar{\omega}_{i}.
$$
The goal is to show that this information guarantees the existence of small enough $\delta\in (0,1/2)$
and $\sigma \in (0,1)$, the latter close to one, both depending only on the data,
such that 
\[
\essinf_{Q((\delta R_i)^p, c_{i} \delta R_i)}\upsilon \geq \bar{\mu}_{i+1}^-,\qquad
\esssup_{Q((\delta R_i)^p, c_{i} \delta R_i)}\upsilon \leq \bar{\mu}_{i+1}^+, 
\]
and
\[
\bar{\mu}_{i+1}^+-\bar{\mu}_{i+1}^- = \bar{\omega}_{i+1} \equiv \sigma \bar{\omega}_{i}.
\]
This will then lead to establishing the induction step in the proof of the H\"older continuity.

For the sake of simplicity in the notation, we will drop the bars from the quantities $\bar{\mu}_{i}^-$, $\bar{\mu}_{i}^+$ and $\bar{\omega}_{i}$ throughout this section. Although they coincide with the ones for equation~\eqref{equation}, we believe that this does not cause any confusion due to the fact that within Case II we work exclusively with equation~\eqref{substituted_equation}. In the end, we translate the information back to $u$.
We shall also denote in short
\[
r :=  \frac{r_{i_0}}{10}= \frac1{80} \left( C_2 2^{\lambda} \right)^{(p-2)/p} R_{i_0} \quad  \textrm{if} \quad i=i_0;  \qquad  r :=  \frac{R_{i}}{10} \quad  \textrm{if} \quad i>i_0.
\]
Observe that then $10Q(r^p,c_{i} r) \subset Q^{i}$ if $i=i_0$ and $10Q(r^p,c_{i-1} r) \subset Q^i$ for all $i > i_0$. We now redefine the $c_i$'s, putting
$$c^\prime_{i_0}  = c_{i_0} ; \qquad c^\prime_i = c_{i-1} \quad  \textrm{if} \quad i>i_0,$$
and drop the prime in the notation.

Inside $Q\left(r^{p},c_ir\right)$, we consider subcylinders of smaller size
$$
Q_{\bar{x},0}\left(r^{p},d_{i}r\right), \qquad d_{i}= \left(\frac{\omega_{i}}{2^{s} \mu_{i}^{-}}\right)^{\frac{p-2}{p}}, \qquad 0<s \leq \lambda,
$$
which are contained in $Q\left(r^{p},2c_i r\right)$ whenever $\bar{x}$ is a point in $B(0,c_i r)$. The number $s$ will be fixed later, depending only on the data.

The reduction of the oscillation in Case II is based on the analysis of an alternative, see~\cite{DiBe93} and ~\cite{Urba08}. For a constant  $\alpha_{0}\in(0,1)$, that will be determined depending only on the data, either
\\

\noindent {\bf The First Alternative.} There exists a cylinder of the type $Q_{\bar{x},0}\left(r^{p},d_{i}r\right)$ for which
\begin{equation}\label{l:c2alt1}
\frac{\nu\left(\left\lbrace(x,t)\in Q_{\bar{x},0}\left(r^{p},d_{i}r\right):\,\upsilon(x,t)<\mu^{-}_{i}+\omega_{i}/2\right\rbrace\right)}{\nu\left( Q\left( r^{p},d_{i}r\right)\right)} \leq \alpha_{0}
\end{equation}
for some $\bar{x} \in B(0,c_i r)$,
or
\\

\noindent {\bf The Second Alternative.} For every cylinder of the type $Q_{\bar{x},0}\left(r^{p},d_{i}r\right)$,
\begin{equation}\label{l:c2alt2}
\frac{\nu\left(\left\lbrace(x,t)\in Q_{\bar{x},0}\left(r^{p},d_{i}r\right):\,\upsilon(x,t)<\mu^{-}_{i}+\omega_{i}/2\right\rbrace\right)}{\nu\left( Q\left( r^{p},d_{i} r\right)\right)}> \alpha_{0}
\end{equation}
holds for all $\bar{x} \in B(0,c_i r)$.
\\

In both cases, we will prove that the essential oscillation of $\upsilon$ within a smaller cylinder decreases in a way that can be quantitatively measured. The parameters $\lambda$ and $s$ will be needed while analyzing the second alternative. The constant $\alpha_0$ will be fixed in the course of the proof of Proposition \ref{p:c2alt1reduction}.

\subsection{Analysis of the First Alternative}

Now we assume that the assumption of the first alternative is satisfied, which claims that there exists a cylinder of the type $Q_{\bar{x},0}\left(r^{p},d_{i}r\right)$  for which~\eqref{l:c2alt1} holds.
We put
\[
\gamma_1^- =\left(\frac{\omega_{i}}{2^{s} \mu_{i}^{-}}\right)^{\frac{p-2}{p}}, \quad \gamma_2^{-} = 1, \quad \eps_- = 2^{-s}, \quad
k_n^- = \mu_i^- + \frac{\eps_- \omega_i}{4} \left(1 + \frac1{2^{n}}  \right),
\]
and observe that, by~\eqref{eq:bar est}, we have
\[
C_1^{p-2} \leq \left(\frac{\mu_i^+}{\mu_i^-}\right)^{p-2} \leq \frac{(\gamma_1^-)^p}{\gamma_2^-}  \left(\frac{k_n^-}{\eps_- \omega_i}\right)^{p-2} \leq 1 \leq C_1^{2-p}.
\]
We may then apply Corollary~\ref{cor:iteration} with $c\equiv C_1^{2-p}$ (see also Remark~\ref{main_lemma_remark}) and find a sufficiently small $\alpha_0$, depending only on the data, such that
\begin{equation}\label{l:c2alt1res}
\upsilon(x,t)\geq \mu_{i}^{-}+2^{-s-2} \omega_{i}, \quad \mbox{for a.e.} \ (x,t)\in Q_{\bar{x},0}\left(\left(r/2\right)^{p},d_{i}r/2\right).
\end{equation}
This step fixes $\alpha_0$ in~\eqref{l:c2alt1} and~\eqref{l:c2alt2}, and from now on we regard it as a universal constant depending only on the data.

We view $Q_{\bar{x},0}\left(\left(r/2\right)^{p},d_{i}r/2\right)$ as a cylinder inside $Q\left(r^{p},2c_ir\right)$. The location of $\bar{x}$ in the ball $B(0,c_i r)$ is only known qualitatively. However, we shall show that the positivity of $\upsilon$, as stated in~\eqref{l:c2alt1res}, spreads over the full ball $B(0,c_i r)$, for all times $-\left(r/8\right)^{p}\leq t \leq 0$. We regard $\bar{x}$ as the center of a larger ball $B(\bar{x},8c_ir)$, which is still contained in $B(0,10c_i r_{i})$, and work within the cylinder
$$
Q_{\bar{x},0}\left(\left(r/2\right)^{p},8c_ir\right) \subset Q^i.
$$

Now we will derive some integral inequalities, which will be used later on, doing the calculations only at a formal level. In particular, the time derivative of $v$ will appear in the sequel and we even assume some continuity properties. For a rigorous justification of the arguments, namely the use of Steklov averages, we refer to~\cite[p. 101--102]{DiBe93} (see also \cite{KuusSiljUrba10}).

Denoting 
$$\mathcal{A}(x,t,v,\nabla v)=c v^{2-p}\norm{\nabla v}^{p-2}\nabla v,$$
where $c=\frac{1}{(p-1)^{(p-1)}}$, we write equation~\eqref{substituted_equation} in the general form
\begin{equation}\label{general}
v_t - \divt\mathcal{A}(x,t,v,\nabla v)=0.
\end{equation}
Since $v^{2-p}$ is bounded from below and from above by the nonzero constants $(\mu_0^{-})^{2-p}$ and $(\mu_0^{+})^{2-p}$, respectively, $\mathcal{A}(x,t,v,\nabla v)$ satisfies conditions~\eqref{structure_1} and \eqref{structure_2} with $\mathcal{A}_0=(\mu_0^{-})^{2-p}$ and $\mathcal{A}_1=(\mu_0^{+})^{2-p}$. Hence, by Lemma~\ref{weak_parabolic}, we have that, for all $k\in \mathbb{R}$, the truncated functions $(v-k)_-$ are weak subsolutions of equation~\eqref{general}, with $\mathcal{A}(x,t,v,\nabla v)$ replaced by $-\mathcal{A}(x,t,k-(v-k)_-,-\nabla (v-k)_-)$. Since
$$
-\mathcal{A}(x,t,k-(v-k)_-,-\nabla (v-k)_-)=c v^{2-p}\norm{\nabla (v-k)_-}^{p-2}\nabla (v-k)_-,
$$
we obtain
\begin{equation}\label{l:weaksolnewcoord}
\begin{split}
&\int\limits_{B(\bar x,8c_i r)} \! \partial_{t}(\upsilon-k)_{-} \eta\,d\mu\,\\
&+\,
c\int\limits_{B(\bar x,8c_i r)} \! \upsilon^{2-p}\norm{\nabla (\upsilon-k)_{-}}^{p-2}\nabla (\upsilon-k)_{-}\cdot \nabla \eta\,d\mu \leq0,
\end{split}
\end{equation}
for all (admissible) test functions $\eta\geq 0$ and for all $t\in(-\left(r/2\right)^{p},0)$.

We now let $\delta \in (0,2^{-(s+2)})$ (this $\delta$ is unrelated to the one defining the sequence of radii and introduced in \eqref{sol}; we use the same notation as, hereafter, there is no cause for misunderstanding) and set
$$
\Phi_{n}(\upsilon)=\int\limits_{0}^{(\upsilon-k_n)_{-}} \! \frac{d\xi}{\left[(1+\delta)(k_n-\mu_{i}^-)-\xi\right]^{p-1}}, \qquad
k_n=\mu_{i}^-+ \delta^n \omega_{i},
$$
and
$$
\Psi_{n}(\upsilon)=\ln{\left(\frac{(1+\delta)(k_n-\mu_{i}^-)}{(1+\delta)(k_n-\mu_{i}^-)-(\upsilon-k_n)_{-}}\right)\geq 0}.
$$
Let $\varphi(x,t)=\varphi_{1}(x)\varphi_{2}(t)$,
where $\varphi_{1}(x)$ is as in~\eqref{poincare_phi}, with $B(x,r)$ replaced by $B(\bar x,8c_i r)$, and
$\varphi_{2}(t) \in [0,1]$ is a smooth function vanishing at $\left\{ -\left(r/2\right)^{p} \right\}$ and such that $\varphi \equiv 1$ in $[-\left(r/4\right)^{p},0]$. Moreover, we assume that $\varphi$ satisfies
\[
0\le \varphi\le 1, \qquad \norm{\nabla \varphi_{1}}\leq \frac{4}{c_ir} \qquad\text{and}\qquad 0 \leq (\varphi_{2})_{t} \leq \frac{4}{r^p}.
\]

The next lemma guarantees bounds for $\Phi_n$ and $\Psi_n$.
\begin{lemma}\label{lem:intineq}
There exists a constant $C$, that can be determined a priori only in terms of the data and $\lambda$, such that
\be\label{l:intineq2}
\begin{split}
\frac{d}{dt}\dashint_{B(\bar x,8c_i r)} \! \Phi_{n}(\upsilon(x,t))\varphi^{p}(x,t)\, d\mu(x) & \\
+\frac1C\frac{\omega_i^{2-p}}{r^p}\dashint_{B(\bar x,8c_i r)}\Psi_{n}^{p}(\upsilon(x,t))\varphi^{p}(x,t) \, d\mu(x) &
\leq
C \ln\left(\frac1\delta\right) \frac{\omega_i^{2-p}}{r^p}.
\end{split}
\ee
\begin{proof}
Denote in short $B \equiv B(\bar x,8c_i r)$. We also drop the arguments $x$ and $t$ from the integrals.
Substitute the test function
$$
\eta:=\frac{\varphi^{p}}{\left[(1+\delta)(k_n-\mu_{i}^-)-(v-k_n)_{-}\right]^{p-1}}
$$
in~\eqref{l:weaksolnewcoord}, with $k=k_n=\mu_{i}^-+ \delta^n \omega_{i}$.
To start with, we estimate
\begin{align*}
\Phi_{n}(\upsilon)&=\int_{0}^{(\upsilon-k_n)_{-}} \! \frac{d\xi}{\left[(1+\delta)(k_n-\mu_{i}^-)-\xi\right]^{p-1}} \\
&= \frac{1}{2-p}\left\lbrace(1+\delta)^{2-p}(k_n-\mu_{i}^-)^{2-p}-((1+\delta)(k_n-\mu_{i}^-)-(\upsilon-k_n)_{-})^{2-p}\right\rbrace \\
&\leq  2 (k_n-\mu_{i}^-)^{2-p} \frac{1}{2-p} \left(1 - \left( \frac{\delta}{1+\delta} \right)^{2-p} \right) \\
&\leq 2 (k_n-\mu_{i}^-)^{2-p} \ln\left( \frac{1+\delta}{\delta} \right) \\
&\leq 4 \omega_i^{2-p} \delta^{n(2-p)} \ln\left( \frac{1}{\delta} \right),
\end{align*}
since $\delta \in (0,1/2)$.
Now, using this, together with the properties of $\varphi$, we can estimate the first term in~\eqref{l:weaksolnewcoord} by
\begin{align*}
\dashint_{B} \! \partial_{t}(\upsilon-k_n)_{-}\eta \, d\mu
&=\dashint_{B} \! \partial_{t}\Phi_{n}(\upsilon) \varphi^{p} \, d\mu\\
&=\frac{d}{dt}\dashint_{B} \! \Phi_{n}(\upsilon)\varphi^{p}\, d\mu-p\dashint_{B} \! \Phi_{n}(\upsilon)\varphi^{p-1}\varphi_{t} \, d\mu\\
&\geq \frac{d}{dt}\dashint_{B} \! \Phi_{n}(\upsilon)\varphi^{p} \, d\mu- C\ln\left( \frac{1}{\delta} \right) \frac{\omega_i^{2-p}}{r^p}.
\end{align*}
For the second term, using Young's inequality and the fact that $\mu_{i}^-\leq \upsilon\leq \mu_{i}^+$, we obtain
\begin{align*}
& \dashint_{B} \! \upsilon^{2-p}\norm{\nabla (\upsilon-k_n)_{-}}^{p-2}\nabla (\upsilon-k_n)_{-}\cdot \nabla \eta \, d\mu\\
&\qquad =p \dashint_{B} \! \upsilon^{2-p}\left(\frac{\varphi}{(1+\delta)(k_n-\mu_{i}^-)-(\upsilon-k_n)_{-}}\right)^{p-1}\\
&\qquad\qquad\qquad \qquad \qquad \times\norm{\nabla (\upsilon-k_n)_{-}}^{p-2}\nabla (\upsilon-k_n)_{-}\cdot \nabla \varphi  \, d\mu\\
& \qquad \quad +\, (p-1)\dashint_{B} \! \upsilon^{2-p}\left(\frac{\norm{\nabla (\upsilon-k_n)_{-}}}{(1+\delta)(k_n-\mu_{i}^-)-(\upsilon-k_n)_{-}}\right)^{p}\varphi^{p} \, d\mu\\
& \qquad \geq \frac{(p-1)}{2} \dashint_{B} \! (\mu_{i}^-)^{2-p}\left(\frac{\norm{\nabla (\upsilon-k_n)_{-}}}{(1+\delta)(k_n-\mu_{i}^-)-(\upsilon-k_n)_{-}}\right)^{p}\varphi^{p} \, d\mu\\
& \qquad \quad -C\dashint_{B} \! (\mu_{i}^+)^{2-p}\norm{\nabla\varphi}^{p} \, d\mu.
\end{align*}
Recalling the definition of $\Psi_{n}$, we have
$$
\nabla \Psi_{n}(\upsilon)=\frac{\nabla (\upsilon-k_n)_{-}}{(1+\delta)(k_n-\mu_{i}^-)-(\upsilon-k_n)_{-}}.
$$
By this fact, together with the conditions on $\varphi$ and~\eqref{eq:bar est}, we obtain
\begin{align*}
\dashint_{B} \! &v^{2-p} \norm{\nabla (\upsilon-k_n)_{-}}^{p-2}\nabla (\upsilon-k_n)_{-}\cdot \nabla \eta \, d\mu\\
&\geq \frac{p-1}{2}(\mu_{i}^-)^{2-p}\dashint_{B} \! \norm{\nabla\Psi_{n}(\upsilon)}^{p}\varphi^{p}\, d\mu
-C \frac{\omega_i^{2-p}}{r^p}.
\end{align*}
Observe here that $(\mu_i^-)^{2-p}/c_i^p = 2^{-\lambda(2-p)} \omega_i^{2-p} \leq \omega_i^{2-p}$.
Altogether, we finally get
\begin{equation}\label{l:intineq1}
\begin{split}
&\frac{d}{dt}\dashint_{B} \! \Phi_{n}(\upsilon)\varphi^{p}\,d\mu
+ \frac{p-1}{2} (\mu_{i}^-)^{2-p}\dashint_{B} \! \norm{\nabla\Psi_{n}(\upsilon)}^{p}\varphi^{p}\,d\mu
\\ & \qquad
 \leq C\ln\left( \frac1\delta \right) \frac{\omega_i^{2-p}}{r^p} ,
\end{split}
\end{equation}
where $C$ depends only on the data.

Next, it follows from~\eqref{l:c2alt1res} that $\Psi_{n}(\upsilon)$ vanishes for all $\norm{x-\bar x}\leq d_{i}r/2$, for $k_n <\mu_{i}^- + 2^{-(s+2)}\omega_{i}$. On the other hand, $\varphi_1 =1$ for all such $x$. Let us denote $A:=\{x\in B :  |\Psi_{n}(\upsilon(x,t))|>0\}$. Now
\[
\begin{split}
\int_{A^c \cap B}\varphi_1^p\, d\mu&
\ge\mu\left(B(\bar x,d_i r/2)\right)\\
&\ge \frac{1}{C 2^{[(2-p)(\lambda-s)/p+4]d_{\mu}}}\mu\left(B\right)\ge \frac{1}{C'}\int_{B} \varphi_1^p \, d\mu,
\end{split}
\]
where $C'>1$ depends on the data and $\lambda$.
Therefore, for all $t\in(-(r/2)^{p},0)$,
\begin{align*}
\int_{A}\varphi_1^p\, d\mu & = \int_{B} \varphi_1^p d\mu -\int_{A^c \cap B}\varphi_1^p\, d\mu
\\ & \leq
\left(1-  \frac{1}{C'} \right)\int_{B} \varphi_1^p d\mu \equiv \gamma \int_{B} \varphi_1^p d\mu,
\end{align*}
for $\gamma \in (0,1)$, depending only on the data and $\lambda$.
So we may apply Corollary~\ref{poincare_corollary}  to obtain
$$
\dashint_{B} \! \norm{\nabla\Psi_{n}(\upsilon)}^{p}\varphi^{p}\,d\mu\geq \frac{1}{C (c_ir)^p}\dashint_{B}\Psi_{n}^{p}(\upsilon)\varphi^{p}\,d\mu,
$$
where $C$ is a constant depending only on the data and $\lambda$. Inserting this into~\eqref{l:intineq1} finishes the proof.
\end{proof}
\end{lemma}

Introduce the quantities
$$
Y_{n}=\sup_{-(r/2)^{p}\leq t\leq0}\dashint_{B(\bar x,8c_i r)\cap \left[\upsilon(x,t)<k_n\right]}\varphi^{p}(x,t) \, d\mu(x), \qquad n=0,1,2,\ldots .
$$
In the next proposition we employ the previous lemma.
\begin{proposition}\label{p:aux}
The number $\alpha \in (0,1)$ being fixed, we can find numbers $\delta$, $\sigma\in(0,1)$, depending only upon the data, $\alpha$, and $\lambda$, such that for $n=0,1,2,\dotso$, either
$$
Y_{n}\leq \alpha
$$
or
$$
Y_{n+1}\leq\max\left\lbrace\alpha;\sigma Y_{n}\right\rbrace.
$$
\end{proposition}
\begin{proof}
From the definition of $Y_{n}$ it follows that, for every $\epsilon\in(0,1)$, there exists $t_{0}\in(-\left(r/2\right)^{p},0)$ such that
\be\label{l:ineqY}
\dashint_{B \cap\left[\upsilon(x,t_{0})<k_{n+1}\right]} \! \varphi^{p}(x,t_{0}) \, d\mu(x) \geq Y_{n+1}-\epsilon.
\ee
We have again denoted in short $B \equiv B(\bar x,8c_i r)$.
The numbers $n\in\mathbb{N}$ and $t_{0}\in(-\left(r/2\right)^{p},0)$ being fixed, we consider the following two cases: either
\be\label{l:a1}
\frac{d}{dt}\left(\dashint_{B} \! \Phi_{n}(\upsilon)\varphi^{p} \, d\mu\right)(t_{0})\geq 0
\ee
or
\be\label{l:a2}
\frac{d}{dt}\left(\dashint_{B} \! \Phi_{n}(\upsilon)\varphi^{p}\, d\mu\right)(t_{0})< 0.
\ee
In both cases we may assume that $Y_{n}>\alpha$, otherwise the proposition becomes trivial.

Assume that~\eqref{l:a1} holds. Then it follows from~\eqref{l:intineq2} that
\be\label{l:a1h}
\dashint_{B}\Psi_{n}^{p}(\upsilon(x,t_{0}))\varphi^{p}(x,t_{0})\,d\mu(x)\leq C\ln \left(\frac{1}{\delta} \right).
\ee
We bound this expression from below by integrating on the smaller set
$$
\left[\upsilon(x,t_{0})<k_{n+1}\right]\cap B,
$$
on which
$$
\Psi_{n}(\upsilon)\geq \ln\l(\frac{1+\delta}{2\delta}\r).
$$
Therefore
\begin{equation*}
\begin{split}
\left[\ln \l(\frac{1+\delta}{2\delta}\r)\right]^{p}\dashint_{B\cap\left[\upsilon(x,t_{0})<k_{n+1})\right]} \! &\varphi^{p}(x,t_{0})\,d\mu(x) \\
&\quad \leq\dashint_{B}\Psi_{n}^{p}(\upsilon(x,t_{0}))\varphi^{p}(x,t_{0})\,d\mu(x).
\end{split}
\end{equation*}
From this,~\eqref{l:a1h} and~\eqref{l:ineqY},
$$
Y_{n+1}\leq \epsilon + C\left[\ln \l(\frac{1}{\delta}\r)\right]^{1-p}.
$$
We choose $\epsilon = \alpha/2$ and $\delta$ so small that
$$
C\left[\ln \l(\frac{1}{\delta}\r)\right]^{1-p}\leq \frac{\alpha}{2},
$$
{\it i.e.}, 
$$
\delta \leq \exp\left(- (2C/ \alpha)^{1/(p-1)}\right) \in(0,2^{-(s+2)})
$$
and the proposition is proved assuming ~\eqref{l:a1} holds.

Assume now ~\eqref{l:a2} holds true and define
$$
t_{*}:=\sup\left\lbrace t\in(-\left(r/2\right)^{p},t_0):\, \frac{d}{dt}\dashint_{B} \! \Phi_{n}(\upsilon(x,t))\varphi^{p}(x,t)\,d\mu\geq 0\right\rbrace.
$$
Rewriting
\begin{equation}\nonumber
\begin{split}
\Phi_{n}(\upsilon(x,t_{*})) = &
\int_{0}^{(\upsilon(x,t_{*})-k_n)_-} \! 
\frac{d\xi}{\left[(1+\delta)(k_n-\mu_{i}^-)-\xi\right]^{p-1}}
\\
= &
\int_{0}^{k_n-\mu_i^-} \!  \chi_{\left[\xi<(\upsilon(x,t_{*})-k_n)_{-}\right]}
\frac{d\xi}{\left[(1+\delta)(k_n-\mu_{i}^-)-\xi\right]^{p-1}}
\\
= &
\left( \delta^n \omega_i\right)^{2-p}\int_{0}^{1} \!  \chi_{\left[\delta^n \omega_i \xi<(\upsilon(x,t_{*})-k_n)_{-}\right]}
\frac{d\xi}{\left(1+\delta-\xi\right)^{p-1}} ,
\end{split}
\end{equation}
Fubini's theorem yields
\begin{equation}\label{PSI-estimate}
\begin{split}
&\dashint_{B} \! \Phi_{n}(\upsilon(x,t_{0}))\varphi^{p}(x,t_{0}) \, d\mu(x)
\\ & \quad \leq \dashint_{B} \! \Phi_{n}(\upsilon(x,t_{*}))\varphi^{p}(x,t_{*}) \, d\mu(x)
\\ & \quad =
\int_{0}^{1} \! \frac{\left( \delta^n \omega_i\right)^{2-p}}{\left(1+\delta-\xi\right)^{p-1}}\left(\dashint_{B\cap
\left[\delta^n \omega_i\xi<(\upsilon(x,t_{*})-k_n)_{-}\right]} \! \varphi^{p}(x,t_{*})\, d\mu(x)\right)d\xi.
\end{split}
\end{equation}
We next estimate from above the integral inside the parenthesis, for each $\xi~\in~(0,1)$. On the one hand, we have, using the definition of $Y_{n}$,
\begin{equation}\nonumber
\begin{split}
\dashint_{B\cap\left[\xi(k_n-\mu_{i}^-)<(\upsilon(x,t_{*})-k_n)_{-}\right]} \! \varphi^{p}(x,t_{*})\,d\mu(x) &
\\   \leq \dashint_{B\cap\left[\upsilon(x,t_{*})<k_n\right]} \! \varphi^{p}(x,t_{*})\,d\mu(x) & \leq Y_{n}.
\end{split}
\end{equation}
On the other hand, from the definition of $t_{*}$ and \eqref{l:intineq2}, we first get
$$
\dashint_{B}\Psi_{n}^{p}(\upsilon)\varphi^{p}(x,t_{*})\,d\mu(x)\leq C \ln \left(\frac{1}{\delta} \right)
$$
and then, integrating over the smaller set
$$
B\cap\left[(\upsilon(x,t_{*})-k_n)_{-}>\xi(k_n-\mu_{i}^-)\right],
$$
we obtain
\[
\begin{split}
& \dashint_{B\cap\left[(\upsilon(x,t_{*})-k_n)_{-}>\xi(k_n-\mu_{i}^-)\right]} \! \varphi^{p}(x,t_{*})\,d\mu(x)
\\ & \qquad \qquad \leq C \ln \left(\frac{1}{\delta} \right) \left[\ln\left(\frac{1+\delta}{1+\delta-\xi}\right)\right]^{-p},
\end{split}
\]
because in this set
$$
\Psi_{n}^{p}(\upsilon)\geq \left[\ln\left(\frac{1+\delta}{1+\delta-\xi}\right)\right]^{p}.
$$
Then, for all $\xi\in(0,1)$,
\[
\begin{split}
& \dashint_{B\cap\left[\xi(k_n-\mu_{i}^-)<(\upsilon-k_n)_{-}\right]} \! \varphi^{p}(x,t_{*})\,d\mu(x)
\\ & \qquad \qquad \leq  \min\left\lbrace Y_{n}, C \ln \left(\frac{1}{\delta} \right) \left[\ln\left(\frac{1+\delta}{1+\delta-\xi}\right)\right]^{-p}\right\rbrace.
\end{split}
\]
Let $\xi_{*}$ be such that $Y_{n}=C\ln \left(\frac{1}{\delta} \right)  \l[\ln\l(\frac{1+\delta}{1+\delta-\xi_{*}}\r)\r]^{-p}$, {\it i.e.},
\[
\begin{split}
\xi_{*}:= &(1+\delta)\left(1 - \exp\left( - \left[\frac{C}{Y_n} \ln \left(\frac{1}{\delta} \right) \right]^{1/p}\right)\right)
\\ \leq & (1+\delta)\left(1 - \exp\left( - \left[\frac{2C}{\alpha} \ln \left(\frac{1}{\delta} \right) \right]^{1/p}\right)\right) =:\sigma_0.
\end{split}
\]
Here we have used the fact that $Y_n > \alpha\geq \alpha/2$.

We will now make an auxiliary assumption: defining $\delta_1 \equiv \delta_1(C,\alpha)$ via
\[
\left[\ln\left(\frac{1+\delta_1}{\delta_1} \right)\right]^p = \frac{2C}{\alpha} \ln \left(\frac{1}{\delta_1}\right),
\]
for which the root $\delta_1$ clearly exists since $p>1$, we obtain
\begin{equation}\label{eq:delta condition 1}
0 < \delta < \min\{1/4,\delta_1\} \quad \Longrightarrow \quad \sigma_0 <1.
\end{equation}
From now on, we assume that~\eqref{eq:delta condition 1} holds.

For $0\leq \xi<\xi_{*}<1$ we have
$$
C\ln \left(\frac{1}{\delta} \right)  \left[\ln\left(\frac{1+\delta}{1+\delta-\xi}\right)\right]^{-p}>Y_{n}
$$
and for $\xi_{*}\leq\xi\leq1$,
\begin{equation} \label{oscar}
C\ln \left(\frac{1}{\delta} \right)  \left[\ln\left(\frac{1+\delta}{1+\delta-\xi}\right)\right]^{-p} \leq Y_{n}.
\end{equation}

Now we are ready to bound the right hand side of~\eqref{PSI-estimate}:
\begin{align*}
&\int_{0}^{1} \! \frac{\left( \delta^n \omega_i\right)^{2-p}}{\left(1+\delta-\xi\right)^{p-1}}\left(\dashint_{B\cap\left[(k_n-\mu_{i}^-)\xi<(\upsilon(x,t_{*})-k_n)_{-}\right]} \! \varphi^{p}(x,t_{*})\,d\mu(x)\right)d\xi
\\ &\quad \leq
\int_{0}^{\xi_{*}} \! \frac{\left( \delta^n \omega_i\right)^{2-p}}{\left(1+\delta-\xi\right)^{p-1}}Y_{n}\,d\xi
\\ & \qquad
+ C\ln \left(\frac{1}{\delta} \right) \int_{\xi_{*}}^{1} \! \frac{\left( \delta^n \omega_i\right)^{2-p}}{\left(1+\delta-\xi\right)^{p-1}}\left[\ln\left(\frac{1+\delta}{1+\delta-\xi}\right)\right]^{-p}\,d\xi
\\ &
\quad = \int_{0}^{1} \! \frac{\left( \delta^n \omega_i\right)^{2-p} Y_{n}}{\left(1+\delta-\xi\right)^{p-1}}\,d\xi
\\ & \qquad
-   \int_{\xi_{*}}^{1} \! \frac{\left( \delta^n \omega_i\right)^{2-p}Y_n}{\left(1+\delta-\xi\right)^{p-1}}
\left\lbrace 1- \frac{C}{Y_n}\ln \left(\frac{1}{\delta} \right)
\left[\ln\left(\frac{1+\delta}{1+\delta-\xi}\right)\right]^{-p}\right\rbrace\,d\xi.
\end{align*}
Our next goal is to obtain an estimate from below, independent of $Y_{n}$, for the second integral on the right hand side of this inequality. Recalling \eqref{oscar}, we have
\begin{align*}
&\int_{\xi_*}^1 \frac{1}{(1+\delta-\xi)^{p-1}}\left\{ 1 - \frac{C}{Y_n} \ln \left(\frac{1}{\delta} \right)
\left[\ln\left(\frac{1+\delta}{1+\delta-\xi}\right)\right]^{-p}\right\} \, d\xi
\\ &\qquad \ge
\int_{\sigma_0}^1 \frac{1}{(1+\delta-\xi)^{p-1}}\left\{ 1 - \frac{C}{\alpha} \ln \left(\frac{1}{\delta} \right)
\left[\ln\left(\frac{1+\delta}{1+\delta-\xi}\right)\right]^{-p}\right\} \, d\xi \\
&\qquad \ge
\left\{1 - \frac{C}{\alpha} \ln \left(\frac{1}{\delta} \right)
\left[\ln\left(\frac{1+\delta}{1+\delta-\sigma_0}\right)\right]^{-p}\right\}
\int_{\sigma_0}^1 \frac{1}{(1+\delta-\xi)^{p-1}} \, d\xi
\\ & \qquad =
\frac12 \int_{\sigma_0}^1 \frac{1}{(1+\delta-\xi)^{p-1}} \, d\xi
\end{align*}
and thus
\begin{align*}
&\int_{0}^{1} \! \frac{\left( \delta^n \omega_i\right)^{2-p}}{\left(1+\delta-\xi\right)^{p-1}}\left(\dashint_{B\cap\left[(k_n-\mu_{i}^-)\xi<(\upsilon(x,t_{*})-k_n)_{-}\right]} \! \varphi^{p}(x,t_{*})\,d\mu(x)\right)d\xi
\\ &\quad \leq
\left(\int_{0}^{\sigma_0} \frac{1}{(1+\delta-\xi)^{p-1}} \, d\xi + \frac12 \int_{\sigma_0}^1 \frac{1}{(1+\delta-\xi)^{p-1}} \, d\xi\right)
\left( \delta^n \omega_i\right)^{2-p}Y_n
\end{align*}
holds.
Therefore, we obtain
\begin{equation}\label{eq:Y(n) lower} \nonumber
\begin{split}
&\dashint_{B} \! \Phi_{n}(\upsilon(x,t_{0}))\varphi^{p}(x,t_{0}) \, d\mu(x) \\
&\leq \frac{\left[(1+\delta) \delta^n \omega_i\right]^{2-p}}{(2-p)}\l(1 -\frac{(1+\delta-\sigma_0)^{2-p}+\delta^{2-p}}{2(1+\delta)^{2-p}}\r)Y_{n}.
\end{split}
\end{equation}
To estimate the left hand side of this inequality from below we integrate over the smaller set $B\cap\left[\upsilon(\cdot,t_{0})<k_{n+1}\right]$ to get, using ~\eqref{l:ineqY},
\begin{equation}\nonumber
\begin{split}
& \dashint_{B} \! \Phi_{n}(\upsilon(x,t_{0}))\varphi^{p}(x,t_{0}) \, d\mu(x)
\\ & \qquad \geq
\dashint_{B\cap\left[\upsilon(x,t_{0})<k_{n+1}\right]} \! \Phi_{n}(\upsilon(x,t_{0}))\varphi^{p}(x,t_{0}) \, d\mu(x)
\\ &\qquad \geq
\Phi_{n}(k_{n+1})  (Y_{n+1}-\epsilon).
\end{split}
\end{equation}
Since 
\begin{align*}
 \Phi_{n}(k_{n+1}) = &\int_{0}^{k_n-k_{n+1}} \! \frac{d\xi}{\left[(1+\delta)(k_n-\mu_{i}^-)-\xi\right]^{p-1}}
 \\ = & \frac{\left[ (1+\delta) \delta^n \omega_i \right]^{2-p} }{(2-p)} \left(1 - \left( \frac{2\delta}{1+\delta} \right)^{2-p}   \right)
\end{align*}
and
$\epsilon$ is arbitrary in $(0,\frac{\alpha}{2})$, we obtain
$$
Y_{n+1}\leq \tilde{\sigma} Y_{n},
$$
where
\[
\tilde{\sigma}=\frac{1 -\frac{(1+\delta-\sigma_0)^{2-p}+\delta^{2-p}}{2(1+\delta)^{2-p}}}{1 - \left( \frac{2\delta}{1+\delta} \right)^{2-p}  }.
\]

We shall next find an upper bound for $\tilde \sigma$ {\em which stays stable as $p \uparrow 2$}.
To start with, by the very definition of $\sigma_0$,
\[
\frac{(1+\delta-\sigma_0)^{2-p}}{(1+\delta)^{2-p}} = \left[\exp\left( - \left[\frac{2C}{\alpha} \ln \left(\frac{1}{\delta} \right) \right]^{1/p}\right)\right]^{2-p}.
\]
Since $p>1$, the equation
\[
\frac{2 C}{\alpha} \ln\left(\frac{1}{\delta_2} \right) =  \left[\ln\left(\frac{1+\delta_2}{8\delta_2} \right)\right]^p
\]
has a solution $\delta_2 \in (0,1/(8e-1))$, depending only on the data and $\alpha$, that remains stable as $p\uparrow 2$.
We have, for $\delta \leq \min\{\delta_1,\delta_2\}=\delta_2$, 
\[
\frac{2 C}{\alpha} \ln\left(\frac{1}{\delta} \right) \leq \left[\ln\left(\frac{1+\delta}{8\delta} \right)\right]^p,
\]
leading first to
\[
-\left[\exp\left( - \left[\frac{2C}{\alpha} \ln \left(\frac{1}{\delta} \right) \right]^{1/p}\right)\right]^{2-p} \leq -\left(\frac{8\delta}{1+\delta} \right)^{2-p},
\]
which implies
\[
-\frac{(1+\delta-\sigma_0)^{2-p}}{(1+\delta)^{2-p}} \leq - \left( \frac{8\delta}{1+\delta}\right)^{2-p},
\]
and therefore also to
\[
\begin{split}
\tilde{\sigma} \leq & \frac{1}{2} \left[\frac{2 -\left(\frac{8\delta}{1+\delta}\right)^{2-p}-\left(\frac{\delta}{1+\delta}\right)^{2-p}}{1 - \left(\frac{2\delta}{1+\delta}\right)^{2-p}}\right]
\\  = &
1 + \frac12 \left(\frac{2\delta}{1+\delta}\right)^{2-p}\left[\frac{2 -4^{2-p}-2^{p-2}}{1 - \left(\frac{2\delta}{1+\delta}\right)^{2-p}}\right].
\end{split}
\]
Applying the mean value theorem to the functions $f(x)=2-4^x-2^{-x}$ and $g(x)=1-(\frac{2\delta}{1+\delta})^x$, in the interval $[0,2-p]$, we obtain 
\[
\begin{split}
& \frac12 \left(\frac{2\delta}{1+\delta}\right)^{2-p}\left[\frac{2 -4^{2-p}-2^{p-2}}{1 - \left(\frac{2\delta}{1+\delta}\right)^{2-p}}\right]
\\ & \qquad
= \frac12 \left(\frac{2\delta}{1+\delta}\right)^{2-p}\left[\frac{- \left(2-p\right)\left[\ln(4)4^{s_1}-\ln(2)2^{-s_1}\right]}{- \left(2-p\right)\left[\ln\left(\frac{2\delta}{1+\delta}\right) \left(\frac{2\delta}{1+\delta}\right)^{s_2}\right]}\right]
\\ & \qquad
= \left[2 \ln\left(\frac{2\delta}{1+\delta}\right)\right]^{-1} \left(\frac{2\delta}{1+\delta}\right)^{2-p-s_2}\left[\ln(4)4^{s_1}-\ln(2)2^{-s_1}\right],
\end{split}
\]
for some $s_1,s_2 \in [0,2-p]$.
But since $\ln(4)4^{s_1}-\ln(2)2^{-s_1}\geq \ln(2)$, for all $s_1\geq 0$, we conclude with the estimate
\[
\begin{split}
\tilde{\sigma} \leq & 1 + \left[2 \ln\left(\frac{2\delta}{1+\delta}\right)\right]^{-1} \left(\frac{2\delta}{1+\delta}\right)^{2-p} \ln(2) =: \sigma <1.
\end{split}
\]
The parameters $\delta$, $\sigma \in (0,1)$, chosen in this way, depend only on the data and $\alpha$, and they are stable as $p\uparrow 2$. This completes the proof.
\end{proof}

We now prove the crucial result towards the expansion to a full cylinder in space.

\begin{lemma}\label{lem:c2alt1m}
For every $\alpha\in(0,1)$, there exists a positive number
$
\delta^{*}\in \left(0,2^{-(s+2)} \right)
$, that can be determined a priori only in terms of the data, $\alpha$, and $\lambda$, such that
$$
\mu\left(\left\lbrace x\in B(\bar x,4c_i r):\,\upsilon(x,t)\leq\mu_{i}^{-}+\delta^{*} \omega_{i}\right\rbrace\right)\leq \alpha\mu\left(B(\bar x,4c_i r)\right),
$$
for all time levels $t\in\left[-\left(r/4\right)^{p},0\right]$.
\begin{proof}
Recall the definition
\[
Y_n=\sup_{-(r/2)^{p}\leq t\leq0}\dashint_{B(\bar x,8c_i r)\cap \left[\upsilon(x,t)<k_n\right]}\varphi^{p}(x,t) \, d\mu(x).
\]
Let $\alpha>0$ be arbitrary and take $\alpha_1 \le \alpha /D_0$, where $D_0$ is the doubling constant. By the previous proposition, we can find $\delta$, $\sigma \in (0,1)$ such that either $Y_{n}\leq \alpha_1$ or $Y_{n+1}\leq\max\left\lbrace\alpha_1;\sigma Y_{n}\right\rbrace$. Iterating, we obtain
$$
Y_{n}\leq\max{\left\lbrace\alpha_1,\sigma^{n}Y_{0}\right\rbrace}, ~~\mbox{}~~ n=0,1,2,\dotso.
$$
Since $Y_{0}\le1$, we have
$$
Y_{n}\leq\max{\left\lbrace\alpha_1,\sigma^{n}\right\rbrace}.
$$
Take $n=n_{0}\in\mathbb{N}$ to be the smallest integer satisfying the condition
$\sigma^{n_{0}}\leq\alpha_1$; then
$$
Y_{n_{0}}\leq\max{\left\lbrace\alpha_1,\sigma^{n_{0}}\right\rbrace}\leq\alpha_1.
$$
On the other hand, we also obtain
$$
Y_{n_{0}}\geq\frac{\mu(\brc{x\in B(\bar x,4c_i r):\,\upsilon(x,t)<k_{n_0}})}{\mu(B(\bar x,8c_i r))},
$$
for all $t\in\left[-\left( r/4 \right)^{p},0\right]$. Therefore,
\begin{eqnarray*}
\mu\left(\brc{x\in B(\bar x,4c_i r):\,\upsilon(x,t)<k_{n_0}}\right) & \leq & \alpha_1 D_0\mu(B(\bar x,4c_i r))\\
& \leq & \alpha \mu(B(\bar x,4c_i r)), 
\end{eqnarray*}
for all $t\in\left[-\left( r/4 \right)^{p},0\right]$, and the lemma follows with $\delta^{*}=\delta^{n_{0}}$.
\end{proof}
\end{lemma}

We are now in a position to prove the main result of this section.

\begin{proposition}\label{p:c2alt1reduction}
Assume that~\eqref{l:c2alt1res} holds for some $\bar{x}\in B(0,c_i r)$. There exists a positive number $\sigma \in (3/4,1)$, depending only on the data and $\lambda$, such that
$$
\upsilon(x,t)\geq \mu_{i}^{-}+(1-\sigma) \omega_i, \quad \mbox{for a.e.} \ (x,t)\in Q\left(\left(r/8\right)^{p},c_ i r\right).
$$
\begin{proof}
Let $\alpha \in (0,1)$ and let $\delta^*$ be as in Lemma~\ref{lem:c2alt1m}, chosen so that
$$
\mu\left(\left\lbrace x\in B(\bar x,4c_i r):\,\upsilon(x,t)\leq\mu_{i}^{-}+\delta^{*} \omega_{i}\right\rbrace\right)\leq \alpha\mu\left(B(\bar x,4c_i r)\right)
$$
holds for all time levels $t\in\left[-\left(r/4\right)^{p},0\right]$.
Our next goal is to apply Lemma~\ref{main_lemma} (see also Remark~\ref{main_lemma_remark}) for 
the choices
\[
\gamma_1^- = 4 c_i = 4\left(\frac{\omega_i}{2^\lambda \mu_i^-} \right)^{(p-2)/p}
, \qquad   \gamma_2^- = 4^{-p} (\delta^*)^{2-p}, \qquad \eps_- =\delta^*
\]
and 
\[
k_n^- = \mu_i^- + \frac{\eps_- \omega_i}{2} \left(1 + \frac1{2^{n}}  \right).
\]
Observe that, due to \eqref{eq:bar est}, the term
\[
\frac{(\gamma_1^-)^p}{\gamma_2^-}  \left(\frac{k_n^-}{\eps_- \omega_i}\right)^{p-2}
= 16^p\left(\delta^* \frac{\omega_i}{2^\lambda \mu_i^-} \frac{k_n^-}{\delta^* \omega_i}  \right)^{p-2}
= 16^p \left( \frac{k_n^-}{2^\lambda \mu_i^-} \right)^{p-2}
\]
is bounded above and below, for a constant depending only on the data and $\lambda$. In particular, the constant is independent of $\delta^*$, a crucial fact in order to avoid vicious circles.
Letting
\[
t^* \in \left(- 4^{-p}(1-(\delta^* )^{2-p})r^p,0\right],
\]
Corollary~\ref{cor:iteration} and Lemma \ref{lem:c2alt1m} imply, by choosing $\alpha \equiv \alpha(data,\lambda)$ small enough, that
\[
v(x,t)\geq \mu_i^- + \frac{\delta^* \omega_i}{4}, \quad \textrm{for a.e} \quad  (x,t) \in B(\bar x, 2c_i r)\times (t^*-\gamma_2^- (r/2)^p,t^*).
\]
Observe carefully that the choice of $\alpha$ also fixes $\delta^*$, and hence $\delta^*$ depends only on the data and $\lambda$.
Since $t^*$ above is arbitrary, the result now follows  with $\sigma = 1 - \delta^*/4$.
\end{proof}
\end{proposition}

As a consequence, we may rephrase the conclusion of the first alternative in the following way.
Here we shall also apply ~\eqref{eq:Case II inclusion} and return to the bar in the notation.
\begin{corollary}\label{c:c2alt1reduction} \label{cor:Case II conclusion 1}
Assume~\eqref{cor:Case II test} is in force and suppose that~\eqref{l:c2alt1} holds for some $\bar x~\in~B(0,c_i r)$. Then there exist $\delta_{II},\sigma_{II} \in (0,1)$, both depending only on the data and $\lambda$, such that
$$
\essosc_{Q\left((\delta_{II} R_{i+1})^{p},c_{i+1} \delta_{II} R_{i+1} \right)} \upsilon\leq \sigma_{II} \bar{\omega}_{i}.
$$
\end{corollary}

\subsection{Analysis of the second alternative}

Now we analyze the second alternative. Assume ~\eqref{l:c2alt2} holds for all cylinders $Q_{\bar{x},0}\left(r^{p},d_{i}r\right)$.  Since
$$
\mu_{i}^{+}-\frac{\omega_{i}}{2}=\mu_{i}^{-}+\frac{\omega_{i}}{2},
$$
we can rephrase~\eqref{l:c2alt2} as
\be\label{l:c2alt2new1}
\frac{\nu\left((x,t)\in Q_{\bar{x},0}\left(r^{p},d_{i}r\right):\,v(x,t)>\mu^{+}_{i}-\omega_{i}/2\right)}{\nu( Q\left( r^{p},d_{i} r\right))}
\leq(1-\alpha_0),
\ee
for all cylinders $Q_{\bar{x},0}\left(r^{p},d_{i}r\right)$ such that $\bar{x} \in B(0,c_i r)$.
In particular, we deduce that there exists a time level $t^{*} \equiv t^*(\bar{x})$, $t^* \in\left(-r^{p},-(\alpha_0/2)r^{p}\right)$, such that
\be\label{l:c2alt2ineq}
\frac{\mu\left(\left\lbrace x\in B(\bar{x},d_i r):\,v(x,t^{*})>\mu_{i}^{+}-\omega_{i}/2\right\}\right)}{\mu(B(\bar{x},d_i r))}
\leq\left(\frac{1-\alpha_0}{1-\alpha_0/2}\right).
\ee
In fact, if~\eqref{l:c2alt2ineq} is violated for all $t\in\left(-r^{p},-(\alpha_0/2)r^{p}\right)$, we would get
\begin{align*}
&\nu\left( \left\lbrace (x,t) \in Q_{\bar{x},0}\left(r^{p},d_{i} r\right):\, v(x,t)>\mu^{+}_{i}-\omega_{i}/2\right\}\right) \\
&\qquad \geq \int_{-r^{p}}^{-(\alpha_0/2)r^{p}} \! \mu\left(\left\lbrace x \in B(\bar{x},d_i r) \ :\ v(x,t)>\mu^{+}_{i}-\omega_{i}/2\right\}\right) \, dt\\
& \qquad > (1-\alpha_0)\nu(Q\left(r^{p},d_{i} r \right)),
\end{align*}
\noindent which contradicts~\eqref{l:c2alt2new1}.

The next lemma asserts that a similar property still holds for all time levels in an interval up to the origin.
We shall now fix the number $s$ in the definition of $d_i$.

\begin{lemma}\label{logarithmic_bound}
There exists $s$, depending only upon the data, such that
\[
\frac{\mu\left(\left\{x\in B(\bar{x},d_i r) \ : \ v(x,t) > \mu_{i}^+ - 2^{-s} \omega_{i} \right\}\right)}{\mu(B(\bar{x},d_i r))}
\le\left(1-\frac{\alpha_0 }{4} \right),
\]
for all  $\bar{x} \in B(0,c_i r)$ and for almost every $t\in\left(-(\alpha_0/2)r^{p},0\right)$.
\begin{proof}
Let
\[
c=\frac{\omega_{i}}{2^{s}}, \qquad k = \mu_{i}^+-\frac{\omega_{i}}{2}
, \qquad
H_k^+= \frac{\omega_{i}}{2} , 
\]
and set $Q:=Q(r^p,d_i r)$ and $B:= B(\bar{x},d_i r)$. Our aim is to use Lemma~\ref{Harnack_logarithm} to forward the information in time. 

Recall the definition
\[
\psi_+(v)=\Psi(H_k^+,(v{-}k)_+,c)=\ln^+\l(\frac{H_k^+}{c+H_k^+-(v{-}k)_+}\r).
\]
On the one hand, since $(v-k)_+ \leq \omega_i/2 \equiv H_k^+$, we have
\[
\psi_+(v)\le \ln\l(\frac{2^{-1} \omega_{i}}{2^{-s} \omega_{i}} \r) = (s-1)\ln 2;
\]
on the other hand, in the set
\[
\{v> \mu_{i}^+- 2^{-s} \omega_{i}\},
\]
we get
\[
\psi_+(v)
\ge 
\ln\l(\frac{2^{-1} \omega_{i}}{2^{-s}\omega_{i}+2^{-s} \omega_{i}}\r) =(s-2) \ln 2.
\]
The last estimate we need is
\[
|\psi_+(v)'|^{2-p}\le \left( \frac{1}{c} \right)^{2-p} = \l(\frac{\omega_{i}}{2^{s}}\r)^{p-2} = (\mu_i^-)^{p-2} d_i^p .
\]

Now let $\varphi\in C_0^\infty(B)$ be a cutoff function which is independent of time
and has the properties $0\le\varphi\le 1$, $\varphi=1$ in $(1-\delta)B$ and
$
|\nabla \varphi|\le (\delta d_i r)^{-1},
$
where $0<\delta<1$ is to be determined later.

Apply Lemma~\ref{Harnack_logarithm} with these choices to conclude
\begin{equation}\notag
\begin{split}
&(s-2)^2 (\ln2)^2 \mu(\{x\in (1-\delta)B :  v(x,t)> \mu_{i}^+- 2^{-s} \omega_{i} \}) \\
&\qquad \le\esssup_{t^*<t<0}\int_{B} \psi_+^2(v)(x,t)\varphi^p(x) \, d\mu \\
&\qquad \le \int_{B}\psi_+^2(v)(x,t^*)\varphi^p(x) \, d\mu \\
&\qquad \quad +C \int_{t^*}^{0}\int_{B}v^{2-p}\psi_+(v)|(\psi_+)^\prime(v)|^{2-p}|\nabla  \varphi|^p \, d\mu\, dt \\
&\qquad \le (s-1)^2 \ln^2(2) \frac{1-\alpha_0}{1-(\alpha_0/2)}\mu(B)+C\frac{(s-1)\ln 2}{\delta^p} \mu(B),
\end{split}
\end{equation}
for almost every $t\in(t^*(\bar x),0)$, where $C$ depends only upon the data. Observe that in the third inequality we have used~\eqref{l:c2alt2ineq} and~\eqref{eq:bar est}.

Now, by the annular decay property \eqref{annular_decay}, we have
\begin{align*}
& \mu(\{x\in B  \ :\ v(x,t)> \mu_{i}^+- 2^{-s} \omega_{i} \}) \\
&\qquad \le\mu(B \setminus (1-\delta) B)+\mu(\{x \in (1-\delta)B) \ : \ v(x,t)> \mu_{i}^+- 2^{-s} \omega_{i}\})\\
&\qquad \le C\delta^\alpha\mu(B)+\mu(\{x \in (1-\delta)B) \ : \ v(x,t)> \mu_{i}^+- 2^{-s} \omega_{i}\}).
\end{align*}
For the first term, we choose $\delta$ small enough so that
\[
C\delta^\alpha<\frac{\alpha_0}{24}
\]
and for the second term we use the previous estimate. Indeed, by choosing $s$ large enough so that
\[
\frac{1-\alpha_0}{1-(\alpha_0/2)} \, \frac{(s-1)^2}{(s-2)^2}
\le 1-\frac{\alpha_0}{3}
\]
and
\[
\frac{C(s-1)}{(\ln 2)\delta^p(s-2)^2} 
\le\frac{\alpha_0}{24}
\]
hold,  we get the claim for almost every $t\in(-(\alpha_0/2)r^{p},0) \subset (t^*(\bar x),0)$.
\end{proof}
\end{lemma}

The information of Lemma~\ref{logarithmic_bound} will be used to show that in a small cylinder about the origin, the solution $v$ is strictly bounded above by
$\mu_{i}^{+}-2^{-m}\omega_i$, for some $m>s$.
We shall also fix $\lambda$, which enters the definition of $c_i$, determining in this process the size of $Q\left(r^{p},c_i r\right)$.

To simplify the notation, from now on we denote $\beta=\frac{\alpha_0}{2}$.

\begin{lemma}\label{choosing_lambda}
For every $\alpha_1\in(0,1)$, there exists $m\geq s$,
depending only on the data and $\alpha_1$,
such that
\[
\frac{\nu\left(\{ (x,t) \in Q_{\bar{x},0}(\beta r^p, d_i r) \ : \ v(x,t) > \mu_{i}^+ - 2^{-m} \omega_{i}  \} \right) }{\nu(Q_{\bar x,0}(\beta r^p,d_i r))}\leq \alpha_1,
\]
for all $\bar x \in B(0,c_i r)$.

\begin{proof}
Let
\[
E_n (t) = \{ x \in B(\bar{x},d_i r) \ : \ v(x,t) > \mu_{i}^+ - 2^{-n} \omega_{i}  \}
\]
and
\[
E_n  = \{ (x,t) \in Q_{\bar{x},0}(\beta r^p, d_i r) \ : \ v(x,t) > \mu_{i}^+ - 2^{-n} \omega_{i}  \}.
\]
Denote
\[
h=\mu_{i}^+- 2^{-(n+1)} \omega_{i}
\]
and
\[
k=\mu_{i}^+- 2^{-n} \omega_{i},
\]
where $n \geq s$ will be chosen large. Here $s$ is as in Lemma~\ref{logarithmic_bound}.
Denote $\sigma B=B(\bar x,\sigma d_i r)$ and $\sigma Q = \sigma B \times (-\beta (\sigma r)^p,0)$, $\sigma>0$.
Let also
\[
w=
\begin{cases}
h-k,\quad &v\ge h, \\
v-k,\quad &k<v<h, \\
0, \quad &v\le k.
\end{cases}
\]
By Lemma~\ref{logarithmic_bound} and the fact that $n\ge s$, we have
\begin{align*}
\mu(x\in B :  w(x,t)=0\})&=\mu(\{x\in B  :  v(x,t)\le k\})
\ge\frac{\alpha_0}{4}\mu(B),
\end{align*}
for almost every $t\in\big(-\beta r^p,0\big)$.
Thus, for almost every $t\in\big(-\beta r^p,0\big)$, we obtain
\[
w_{B}(t)=\dashint_{B \times\{t\}} w \, d\mu \le \l(1-\frac{\alpha_0}{4}\r)(h-k)
\]
and, consequently,
\[
h-k-w_{B}(t) \ge \frac{\alpha_0}{4}(h-k).
\]

Using the $(q,q)$-Poincar\'e inequality for some $q<p$ (see~\eqref{poincare} and the subsequent remarks), yields
\begin{align*}
(h-k)^q&\mu(E_{n+1}(t))\le\l(\frac{4}{\alpha_0}\r)^q\int_{B\times\{t\}}|w-w_{B}(t)|^q\, d\mu \\
&\le
C(d_i r)^q\int_{B \times\{t\}} |\nabla w|^q \, d\mu = C(d_i r)^q\int_{E_n(t)\setminus E_{n+1}(t)} |\nabla v|^q \, d\mu,
\end{align*}
for almost every $t\in\big(-\beta r^p,0\big)$. The constant $(4/\alpha_0)^q$ above was absorbed into the constant $C$. Now we integrate the above inequality over time to get
\[
(h-k)^q\nu(E_{n+1})\le C (d_ir)^q\int_{E_n\setminus E_{n+1}} |\nabla v|^q \, d\nu.
\]
Next, we introduce a cutoff function $\varphi \in C^\infty(2Q)$ such that
$0\le\varphi\le 1$, $\varphi=1$ in $Q$, $\vp$ vanishes on the parabolic boundary of $2Q$,
and
\[
|\nabla \varphi|\le\frac{C}{d_i r} \quad\text{and}\quad \left(\frac{\partial \varphi}{\partial t}\right)_+\le \frac{C}{\beta r^p}.
\]
Now, H\"older's inequality gives
\begin{align*}
& (h-k)^q\nu(E_{n+1})
\\ & \qquad  \le C(d_ir)^q\l(\int_{E_{n}\setminus E_{n+1}} |\nabla v|^p \, d\nu\r)^{q/p}\nu(E_{n}\setminus E_{n+1})^{1-q/p}
\\ & \qquad \le C\l((d_ir)^p \int_{Q(2\beta r^p,2d_i r)} |\nabla (v{-}k)_+|^p\varphi^p \, d\nu\r)^{q/p}\nu(E_{n}\setminus E_{n+1})^{1-q/p}.
\end{align*}
Since $n \ge s$,
the first factor on the right hand side can be estimated by Lemma~\ref{substituted_energy} as
\begin{equation}\nonumber
\begin{split}
&\int_{2Q} |\nabla (v{-}k)_+|^p\varphi^p \, d\nu \\
&\le C\int_{2 Q}(v{-}k)_+^p|\nabla\varphi|^p\, d\nu +C(\mu_{i}^-)^{p-2}\int_{2Q} (v{-}k)_+^2 \left(\frac{\partial \varphi}{\partial t}\right)_+\, d\nu \\ 
& \le C\l(\frac{1}{(d_i r)^p}\l(\frac{\omega_{i}}{2^n}\r)^p+\frac{(\mu_{i}^-)^{p-2}}{\beta r^p}\l(\frac{\omega_{i}}{2^n}\r)^2\r)\nu\left(2Q\right)
\\ &\le \frac{C}{(d_i r)^p}\left(\frac{\omega_{i}}{2^n}\right)^p\nu\left(Q\right).
\end{split}
\end{equation}
In the last inequality we have used the doubling property of the measure $\nu$, together with the estimate
\[
(\mu_{i}^-)^{p-2}\l(\frac{\omega_{i}}{2^n}\r)^{2}\le \l(\frac{\omega_{i}}{2^{s}\mu_{i}^-}\r)^{2-p}\l(\frac{\omega_{i}}{2^{n}}\r)^{p}= \frac{1}{d_i^p}\l(\frac{\omega_{i}}{2^{n}}\r)^{p}.
\]
We obtain
\[
\l(\frac{\omega_{i}}{2^{n+1}}\r)^q
\nu(E_{n+1})\le C
\l(\frac{\omega_{i}}{2^n}\r)^q
\nu\left(Q\right)^{q/p}\nu(E_{n}\setminus E_{n+1})^{1-q/p}.
\]
Finally, summing $n$ over $s,\dots,m-1$, gives
\[
(m-s)\nu(E_{m})^{p/(p-q)}
\le C\nu\left(Q\right)^{q/(p-q)}\nu\left(Q\right) = C\nu\left(Q\right)^{p/(p-q)},
\]
and hence
\[
\nu(E_{m})\le \frac{C}{(m-s)^{(p-q)/p}}\nu\left(Q\right).
\]
Choosing $m$ large enough finishes the proof.
\end{proof}
\end{lemma}

We now cover $Q(\beta r^p, c_i r)$ with smaller cylinders $Q\left( \beta r^p,d_i r\right)$ to obtain the next result.

\begin{corollary} \label{cor:second alt}
For every $\alpha_2 \in (0,1)$, there exists $m>1$, depending only on the data and $\alpha_2$, such that
\[
\frac{\nu\left(\{ (x,t) \in Q\l(\beta r^p,c_i r\r) \ : \ v(x,t) >
\mu_{i}^+- 2^{-m}\omega_{i}  \}\right)}{\nu\left(Q(\beta r^p,c_ir)\right)} \le \alpha_2 .
\]
\begin{proof}
Take a dense subset $\{x_j\}$ of $B(0,c_i r)$ and let $B_j := B(x_j,d_i r/5)$. By Vitali's covering theorem, we find disjoint balls $\{B_{j_k}\}_{k\in \N}$
such that 
$$B(0,c_i r) \subset \bigcup_{k \in \N} 5B_{j_k}.$$
Let $Q_k := 5B_{j_k} \times (-\beta r^p,0)$. By Lemma~\ref{choosing_lambda}, with $\alpha_1 = D_0^{-4} \alpha_2$, we have
\[
\nu\left(\{ (x,t) \in Q_k  :  v(x,t) > \mu_{i}^+ - 2^{-m} \omega_{i}  \} \right) \leq D_0^{-4} \alpha_2  \nu(Q_k),
\]
for a suitably large $m$, independent of $k$. It follows that
\begin{align*}
& \nu\left(\{ (x,t) \in Q\l(\beta r^p,c_i r\r) \ : \ v(x,t) > \mu_{i}^+- 2^{-m}\omega_{i}  \}\right)
\\ & \qquad \leq  \nu\left(\{ (x,t) \in \cup_k  Q_k \ : \ v(x,t) > \mu_{i}^+ - 2^{-m} \omega_{i}  \} \right)
\\ & \qquad \leq  \sum_k \nu\left(\{ (x,t) \in  Q_k \ : \ v(x,t) > \mu_{i}^+ - 2^{-m} \omega_{i}  \} \right)
\\ & \qquad \leq D_0^{-4} \alpha_2 \sum_k  \nu(Q_k)
\\ & \qquad \leq D_0^{-1} \alpha_2 \sum_k  \nu(B_{j_k} \times (-\beta r^p,0) )
\\ & \qquad \leq D_0^{-1} \alpha_2 \nu\left(B(0,2c_i r) \times (-\beta r^p,0) \right),
\\ & \qquad \leq  \alpha_2 \nu\left(Q\l(\beta r^p,c_i r\r)\right),
\end{align*}
where we have used the doubling property of $\mu$ repeatedly. The estimate concludes the proof.
\end{proof}
\end{corollary}

\begin{proposition} \label{prop:c2alt2_main_result}
Assume that~\eqref{l:c2alt2} holds for all $\bar{x}\in B(0,c_i r)$. There exists a positive number $\sigma \in (0,1)$, depending only on the data, such that
$$
\upsilon(x,t)\leq \mu_{i}^{+}- (1-\sigma) \omega_i, \quad \mbox{for a.e.} \ (x,t)\in Q\left(\beta \left(r/2\right)^{p},c_i r/2\right).
$$
\begin{proof}
As in the conclusion of the First Alternative, we shall apply Corollary~\ref{cor:iteration} (see also Remark~\ref{main_lemma_remark}). Choose the parameters
\[
\gamma_1^+ = c_i = \left(\frac{\omega_i}{2^\lambda \mu_i^-} \right)^{(p-2)/p}
, \qquad   \gamma_2^+ = \frac{\alpha_0}{2}, \qquad \eps_+ =2^{-m},
\]
where $m>1$ is the number obtained in the previous corollary. Define
\[
k_n^+ = \mu_i^+ - \frac{\eps_+ \omega_i}{2} \left(1 + \frac1{2^{n}}  \right),
\]
and put $\lambda=m$. Due to \eqref{eq:bar est}, 
the quantity
\[
\frac{(\gamma_1^+)^p}{\gamma_2^+}  \left(\frac{k_n^+}{\eps_+ \omega_i}\right)^{p-2}
= \frac{2}{\alpha_0} \left( \frac{\omega_i}{2^\lambda \mu_i^-} \frac{k_n^+}{2^{-m} \omega_i}  \right)^{p-2}
= \frac{2}{\alpha_0}  \left( \frac{k_n^+}{\mu_i^-} \right)^{p-2}
\]
is bounded above and below, for a constant depending only on the data. Therefore,
Corollary~\ref{cor:iteration} is applicable and making use of
Corollary~\ref{cor:second alt} the result follows by choosing $\sigma = 1-2^{-m-1}$.
\end{proof}
\end{proposition}

We now fix the size of the cylinder and finish the analysis of the Second Alternative, redefining
\begin{equation}\label{eq:lambda chosen}
\lambda = \max\{\log_2(2H+1),m\}.
\end{equation}  
The choice trivially satisfies~\eqref{lambda large}.

Combining the above Proposition with~\eqref{eq:Case II inclusion} and returning to the bar in the notation, we obtain the next result.

\begin{corollary} \label{cor:Case II conclusion 2}
Suppose that the Second Alternative holds. Then
there are positive numbers $\delta_{III},\sigma_{III} \in (0,1)$, both depending only on the data, such that
$$
\essosc_{Q\left((\delta_{III} R_{i+1})^{p},c_{i+1} \delta_{III} R_{i+1}  \right)}\upsilon\leq \sigma_{III} \bar{\omega}_{i}.
$$
\end{corollary}

We finally prove the H\"older continuity of $u$. Theorem~\ref{main_theorem} is an immediate consequence of the following theorem.
\begin{theorem}
Suppose that $u$ is a nonnegative weak solution of  equation~\eqref{equation} in $Q_{x,t}(R^p,R)$. Then there are
positive constants $C$ and $\alpha$, both depending only on the data, such that
\[
\essosc_{Q_{x,t}(\varrho^p,\varrho)}{u} \leq C \left(\frac{\varrho}{R} \right)^{\alpha} \esssup_{Q_{x,t}(R^p,R)}{u}
\]
for all $0<\varrho<R$. The constants are stable as $p \uparrow 2$.
\begin{proof}
After translation, we may assume that $(x,t) \equiv (0,0)$.
We shall combine Corollaries~\ref{cor:Case I conclusion},~\ref{cor:Case II conclusion 1}
and~\ref{cor:Case II conclusion 2}. Indeed, take $\delta := \min\{\delta_I,\delta_{II},\delta_{III}\}$ and
$\sigma := \max\{\sigma_I,\sigma_{II},\sigma_{III}\}$. Then, we have
\[
\essosc_{Q((\delta^i R)^p ,\delta^i R)}{u} \leq \omega_i := \sigma^i \omega_0, \qquad i=0,1,\ldots,i_0
\]
and
\[
\essosc_{Q((\delta^i R)^p ,c_{i-1} \delta^i R)}{\upsilon} \leq \sigma^{i-i_0} \bar{\omega}_{i_0}, \qquad i > i_0 .
\]
Observe that, since $p\leq 2$, we have
$$
 \bar{\omega}_{i_0} \leq  (\mu_{i_0}^-)^{p-2} \omega_{i_0} .
$$
In view of the mean value theorem and~\eqref{Harnack_assumption_u}, this implies
\begin{equation}\nonumber
\begin{split}
\essosc_{Q((\delta^{i+1} R)^p ,c_i \delta^{i+1} R)}{u} \leq &  (\bar{\mu}_{i+1}^+)^{1/(p-1)} - (\bar{\mu}_{i+1}^-)^{1/(p-1)}
\\ \leq &
\frac1{p-1} (\bar{\mu}_{i+1}^+)^{(2-p)/(p-1)}\left(\bar{\mu}_{i+1}^+ - \bar{\mu}_{i+1}^- \right)
\\ \leq &
\frac1{p-1} (\bar{\mu}_{i_0}^+)^{(2-p)/(p-1)}\sigma^{i-i_0+1} \bar{\omega}_{i_0}
\\ \leq &
\frac{1}{p-1} \left(\frac{\mu_{i_0}^+}{\mu_{i_0}^-}\right)^{2-p} \sigma^{i-i_0+1} \omega_{i_0} 
\\ \leq & C \sigma^{i+1} \omega_0 \equiv C \omega_{i+1}.
\end{split}
\end{equation}
for all $i \geq i_0$, with $C$ depending only on the data. But since $c_i \geq 1$ (recall the choice of $\lambda$ in~\eqref{eq:lambda chosen} and \eqref{c(i) large}) we finally obtain
\[
\essosc_{Q((\delta^i R)^p , \delta^i R)}{u} \leq C \sigma^i \omega_0, \qquad i=0,1,2,\ldots .
\]
From this, the result follows in a standard way (cf. \cite{DiBe93, Urba08}).
\end{proof}
\end{theorem}

\bibliography{citations}
\bibliographystyle{plain}

\bigskip
\noindent Addresses:

\noindent T.K.: Aalto University, Institute of Mathematics, P.O. Box 11100, FI-00076 Aalto, Finland. \\
\noindent
E-mail: {\tt tuomo.kuusi@tkk.fi}\\

\noindent R.L.: CMUC, Department of Mathematics,
University of Coimbra, 3001-454 Coimbra, Portugal. \\
\noindent
E-mail: {\tt rojbin@mat.uc.pt}\\

\noindent J.S.: Aalto University, Institute of Mathematics, P.O. Box 11100, FI-00076 Aalto, Finland. \\
\noindent
E-mail: {\tt juhana.siljander@tkk.fi}\\

\noindent J.M.U.: CMUC, Department of Mathematics,
University of Coimbra, 3001-454 Coimbra, Portugal. \\
\noindent
E-mail: {\tt jmurb@mat.uc.pt}\\

\end{document}